\documentclass[12pt]{article}
\RequirePackage{microtype}
\RequirePackage{etoolbox}
\RequirePackage{booktabs}
\RequirePackage{refcount}
\RequirePackage{totpages}
\RequirePackage{environ}

\RequirePackage{bbm}

\usepackage[dvipsnames]{xcolor}
\usepackage{fullpage,float,amsmath,amsfonts,amssymb,amsthm,thmtools,mathtools,bm}
\usepackage{mathrsfs}
\usepackage{tikz}
\usepackage{enumitem}
\usepackage{caption}
\usepackage{hyperref}

\usetikzlibrary{positioning}
\usetikzlibrary{shapes.geometric}
\usetikzlibrary{arrows.meta}
\usepackage{tikz-cd} 

\usepackage{hyperref}
\hypersetup{linktocpage}
\hypersetup{
    colorlinks,
    citecolor=red,
    filecolor=black,
    linkcolor=blue ,
    urlcolor=red
}

\usepackage{algorithm}
\usepackage{chngcntr}
\counterwithin{algorithm}{section}
\usepackage[noend]{algpseudocode}
\algnewcommand{\LeftComment}[1]{\Statex \(\triangleright\) #1}
\algnewcommand\algorithmicassumptions{\sc{Assumptions:}}
\algnewcommand\Assumptions{\item[\algorithmicassumptions]}
\algrenewcomment[1]{\hfill \(\triangleright\) \emph{\small #1}} 
\algnewcommand{\algorithmicand}{\textbf{and}}
\algnewcommand{\algorithmicor}{\textbf{or}}
\algnewcommand{\FOR}{\algorithmicfor}
\algnewcommand{\OR}{\algorithmicor}
\algnewcommand{\AND}{\algorithmicand}
\algnewcommand{\IF}{\algorithmicif}
\algnewcommand{\THEN}{\algorithmicthen}
\algnewcommand{\ELSE}{\algorithmicelse}
\algnewcommand{\Fail}{\textsc{Fail}}
\algnewcommand{\Cert}{\textsc{Cert}}
\algnewcommand{\NoCert}{\textsc{NoCert}}
\algnewcommand{\Flag}{\textsc{Flag}}
\algnewcommand{\CommentLine}[1]{\(\triangleright\) \emph{\small #1}}
\algnewcommand{\InlineFor}[2]{\algorithmicfor\ #1\ \algorithmicdo\ #2} 
\algnewcommand{\InlineIf}[2]{
  \algorithmicif\ #1\ \algorithmicthen\ #2}
\algnewcommand{\InlineIfElse}[3]{
  \algorithmicif\ #1\ \algorithmicthen\ #2\ \algorithmicelse\ #3}

\makeatletter
\let\original@algocf@latexcaption\algocf@latexcaption
\long\def\algocf@latexcaption#1[#2]{%
  \@ifundefined{NR@gettitle}{%
    \def\@currentlabelname{#2}%
  }{%
    \NR@gettitle{#2}%
  }%
  \original@algocf@latexcaption{#1}[{#2}]%
}
\makeatother

\makeatletter
\newcounter{algorithmicH}
\let\oldalgorithmic\algorithmic
\renewcommand{\algorithmic}{%
  \stepcounter{algorithmicH}
  \oldalgorithmic}
\renewcommand{\theHALG@line}{ALG@line.\thealgorithmicH.\arabic{ALG@line}}
\makeatother

\algrenewcommand\Call[2]{\nameref{#1}\ifthenelse{\equal{#2}{}}{}{\ensuremath{(#2)}}}%

\newcommand{\algoCaptionLabel}[2]{%
     \caption[\textproc{#1}]{\textproc{#1}\ifthenelse{\equal{#2}{}}{}{$(#2)$}}%
     \label{algo:#1}%
     }%

\usepackage[capitalize]{cleveref}

\def\mA{{\bm A}}
\def\ma{{\bm a}}
\def\mB{{\bm B}}

\def\mX{{\bm X}}
\def\vx{{\bm x}}

\def\P{\mathbb{P}}
\def\A{\mathbb{A}}

\def\F{\mathbb{F}}
\def\K{\mathbb{K}}
\def\Kbar{\overline{\mathbb{K}}{}}

\def\C{\mathbb{C}}
\def\L{\mathbb{L}}

\def\Q{\mathbb{Q}}

\def\Z{\mathbb{Z}}
\def\N{\mathbb{N}}

\def\sfA{\mathsf{A}}
\def\sfP{\mathsf{P}}
\newcommand{\nondef}{{{\textsf{undefined}}}}
\def\coeff{\mathrm{coeff}}
\def\acoeff{\mathbf{coefficients}}
\def\lc{\mathrm{lc}}

\providecommand{\keywords}[1]
{
  \smallskip\noindent\small	
  \textbf{\textbf{Keywords---}} #1
}

\declaretheorem[style=plain,parent=section]{definition}
\declaretheorem[sibling=definition]{theorem}

\declaretheorem[sibling=definition]{proposition}
\declaretheorem[sibling=definition]{lemma}
\declaretheorem[sibling=definition]{fact}
\declaretheorem[style=remark,sibling=definition,qed={\qedsymbol}]{remark}

 

\title{Primes of bad reduction for systems of polynomial equations}

\author{
    Jesse Elliott\footnotemark[1]\hspace{2mm}\footnotemark[2]
    \quad
    Éric Schost\footnotemark[2]
}

\begin{document}
\maketitle

\footnotetext[1]{Sorbonne Université, Inria, IMJ--PRG, Paris, France.
Email: \texttt{jesse.elliott@inria.fr}.}

\footnotetext[2]{University of Waterloo, David R. Cheriton School of Computer Science, Waterloo, Canada.
Emails: \texttt{jakellio@uwaterloo.ca}, \texttt{eschost@uwaterloo.ca}.}




\begin{abstract}
  Consider polynomials $F_1,\dots,F_s$ in $\K[X_1,\dots,X_n]$ over a
  field $\K$, their zero-set $V(F_1,\dots,F_n)$ in $\Kbar^n$ and its
  decomposition into equidimensional components $V_0,\dots,V_n$ (with
  $V_i$ either empty or of dimension $i$ for all $i$). To each $V_i$, we can
  associate its Chow forms, which are polynomials in new variables
  $(U_{k,j})_{0\le k\le i, 0 \le j \le n}$, uniquely defined up to a
  scalar factor. These Chow forms completely characterize $V_i$: we
  can recover equations for $V_i$ from them, and their degree is
  $(i+1)$ times the degree of $V_i$.

  We discuss the situation when the $F_i$'s have integer coefficients,
  and study the question of when the Chow forms of the $V_i$'s defined
  as above can be reduced modulo $p$ to give Chow forms of the
  equidimensional components of $V(F_1 \bmod p,\dots,F_s \bmod p)$. We
  show that this is the case as soon as $p$ does not divide a certain
  nonzero integer $\Delta$ of height $O(n^{14} s h d^{3n+4})$, with $d$
  and $h$ bounds on respectively the degrees and heights of the
  $F_i$'s.

\end{abstract}

\section{Introduction}

Consider polynomials $F_1,\dots,F_s$ with integer coefficients; they
define an algebraic set $V=V(F_1,\dots,F_s)$ in $\C^n$. We make no
regularity assumption on these polynomials, and in particular we do
not assume that $V$ is equidimensional. As a result, we can write the
equidimensional decomposition of $V$ as $V = V_0 \cup \cdots \cup
V_{n-1} \cup V_n$, where for all $i$, $V_i$ is what we will call the
$i$th equidimensional component of $V$, that is, the union of its
irreducible components of dimension $i$ (it may be empty).

Our purpose in this paper is to quantify primes of ``bad reduction''
for $V$. One can define ``good'' and ``bad'' primes in several
possible manners. For instance, one could ask that the Gr\"obner basis
of $(F_1,\dots,F_s)$ reduces modulo $p$ to that of $(F_1 \bmod
p,\dots,F_s \bmod p)$. Our definition will be of a geometric nature:
we want to express that ``reducing'' $V_0, \dots, V_{n-1}, V_n$ modulo
$p$ gives the equidimensional decomposition of $V(F_1 \bmod
p,\dots,F_s \bmod p) \subset \overline{\F_p}^n$. This will imply in
particular that for all $i$, the $i$th equidimensional components of
$V$ and of $V(F_1 \bmod p,\dots,F_s \bmod p)$ have the same degree,
but our requirements will be stronger. Note that even in the simple
case where $V(F_1,\dots,F_s)$ has dimension zero, $V(F_1 \bmod
p,\dots,F_s \bmod p)$ may have components of arbitrary dimensions
(they give rise to primary components of the ideal $\langle
F_1,\dots,F_s\rangle \subset \Z[X_1,\dots,X_n]$ that contain
the prime $p$).

An equidimensional algebraic set $Y$ can be described by its {\em Chow
  forms}: if $Y$ defined over a field $\K$ and has dimension $r$, its
Chow forms are squarefree polynomials, {\it a priori} in $\Kbar[\bm
  U_0,\dots,\bm U_r]$, where $\bm U_0,\dots,\bm U_r$ are new
indeterminates, with $\bm U_i = U_{i,0},\dots,U_{i,n}$ for all
$i$. All these polynomial differ by a nonzero multiplicative factor in
$\Kbar$, and since $Y$ is defined over $\K$, it admits Chow forms with
coefficients in $\K$ as well. If $Y$ is empty, its Chow forms are by
convention the constants $c \in \Kbar-\{0\}$.


We will review the definition of Chow forms in Section~\ref{sec:chow};
for the moment, we point out a few noteworthy properties.  First, the
Chow forms of $Y$ are multi-homogeneous in $\bm U_0,\dots,\bm U_r$,
their degree in any of the blocks of variables $\bm U_i$ being the
degree of $Y$.  Next, given a Chow form of $Y$, one can recover
defining equations for it (see~\cite[Corollary~3.2.6]{GeKaZe94}
and~\cite{BlJeSo04}; we point out a simple derivation of this claim in
Section~\ref{sec:proj}). Relatedly, and importantly for possible
applications of our main result, if $Y$ has dimension zero one can
derive from any of its Chow forms a description of $Y$ by means of
univariate polynomials $(Q,V_1,\dots,V_n)$ in $\K[T]$, of the form
\[Y =\left \{ \left (\frac{V_1(\tau)}{Q'(\tau)},\dots,\frac{V_n(\tau)}{Q'(\tau)}\right ) \ \mid \ Q(\tau) = 0\right \}.\]
Indeed, when $Y$ has dimension $r=0$, its Chow forms are polynomials
in variables $U_{0,0},\dots,U_{0,n}$, and $Q$, resp.\ $V_1,\dots,V_n$,
above are obtained by evaluating a Chow form of it, resp.\ its partial
derivatives, at $T,\alpha_1,\dots,\alpha_n$ for some
$\alpha_1,\dots,\alpha_n$ in $\K$ (this goes back
to~\cite{Kronecker82} and~\cite{Macaulay16}).

In any dimension, when $Y$ is defined over $\K=\Q$, it admits Chow
forms with rational coefficients. Clearing denominators and factoring
out common factors, we deduce that $Y$ has exactly two {\em primitive}
Chow forms, that is, Chow forms with integer coefficients and contents
$1$ (the contents of a nonzero polynomial with integer coefficients is
the GCD of its set of coefficients). In particular, these polynomials
can be reduced modulo any prime $p$.

Getting back to our situation with $V=V(F_1,\dots,F_s)$, our
definition of a good prime $p$ is the following: $p$ is a
good prime if and only if the reductions modulo $p$ of the primitive
Chow forms of the equidimensional components of $V(F_1,\dots,F_s)$ are
Chow forms of the equidimensional components of $V(F_1 \bmod
p,\dots,F_s \bmod p)$. Since we are considering primitive
multi-homogeneous polynomials, their degrees do not drop modulo any
prime $p$, so that for a good prime $p$ the $i$th equidimensional
components of $V(F_1,\dots,F_s)$ and of $V(F_1 \bmod p,\dots,F_s \bmod
p)$ indeed have the same degree.

The last ingredient we need in order to state our main result is the
{\em height}, first of a nonzero integer, then of a nonzero polynomial
with integer coefficients: the {\em height} $h(c)$ of a nonzero
integer $c$ is the natural logarithm $\ln(|c|)$, and the height $h(P)$
of a nonzero polynomial $P$ over $\Z$ is the maximum of the heights of
its nonzero coefficients. Later on, we will extend this definition to
polynomials with rational coefficients.

\begin{theorem}\label{theo:main}
  Let $F_1,\dots,F_s$ be in $\Z[X_1,\dots,X_n]$, with degree at most
  $d$ and height at most $h$. There exists a nonzero integer $\Delta$ with
  \[ h(\Delta) \in O(n^{14} s h d^{3n+4} )\]
  such that if a prime $p$ does not divide $\Delta$, then the
  reductions modulo $p$ of the primitive Chow forms of the
  equidimensional components of $V(F_1,\dots,F_s)$ are Chow forms of
  the equidimensional components of $V(F_1 \bmod p,\dots,F_s \bmod
  p)$.
\end{theorem}

This result was motivated by our previous work~\cite{ElSc25,El2025,ElSc24}. In the
context of modular algorithms based on the Chinese Remainder Theorem
(for instance, algorithms for solving polynomial systems), previous
work by Pernet~\cite{Pernet14} and B{\"o}hm {\it et
  al.}~\cite{Bhm2012TheUO} introduced the idea of allowing a number of
unlucky primes (for which we are not given the correct modular image),
and using error correction. In~\cite{ElSc25,El2025,ElSc24}, we gave a quantitative
analysis of this approach, highlighting the need for an {\it a priori}
bound on the number of unlucky primes, and we illustrated this
idea with the example of algorithms for solving polynomial
systems. For this question, we relied on a previous result by D'Andrea
{\it et al.}~\cite{d2019reductions}, who gave a bound similar to the
one in our main theorem, but that applies only when $V(F_1,\dots,F_s)$
has dimension zero.

Extending the result of~\cite{d2019reductions} to arbitrary $F_i$'s is
a natural question, but the techniques we use to prove our main
theorem are not an extension of those in that
reference. Instead, we will describe an algorithm to compute the Chow
forms of the equidimensional components of $V(F_1,\dots,F_s)$, and
quantify the primes $p$ for which we can reduce all its steps modulo
$p$. This algorithm is inspired by a previous one
from~\cite{JeKrSaSo04}, and borrows many of its basic elements from
that reference (as well as from~\cite{BlJeSo04}), but it is somewhat
simpler: the algorithm in~\cite{JeKrSaSo04} uses more advanced
techniques than ours (such as Newton iteration), in order to keep
runtime under control. Our algorithm is not expected to be efficient
by any means, but it will be rather straightforward to analyze for
what primes it degenerates, since it is based almost entirely on
polynomial GCDs.

We also mention the work of Tsigaridas et
al.~\cite{TsigaridasErgurDogan2023}, which gives the first
bit-complexity bound for computing Chow forms (and Hurwitz forms)
using a deterministic resultant-based algorithm. However, we do not
use this result in our work.

\begin{remark}
For readability, we state Theorem~\ref{theo:main} in asymptotic form,
since the explicit inequalities quickly become lengthy and difficult
to parse. We also provide a Maple script that generates the
full inequality (see \href{https://github.com/Jesse-Allister-Kasien-Elliott/Primes-of-bad-reduction-for-systems-of-polynomial-equations}{repository}). In
Section~\ref{sec:experiments}, we compare this bound with the
inequality from~\cite{d2019reductions}, which quantifies for primes
$p$ where polynomial system known to be zero-dimensional over $\Q$ has
a different number of solutions modulo $p$.
\end{remark}


\section{On polynomials with rational coefficients}


\subsection{Height}\label{sec:defH}

We defined the height of integers and polynomials with integer
coefficients in the introduction. We can extend the definition to
polynomials with rational coefficients. For this section, our main
reference is~\cite{KrPaSo01}.

\begin{definition}\label{def:height} 
  The {\em height} $h(F)$ of a nonzero polynomial $F$ with
  coefficients in $\Q$ is the maximum of the heights of its minimal
  denominator $d$, and of the heights of the nonzero coefficients of
  $d F$ (which are all integers).
\end{definition}
The following facts about heights of polynomials with integer
coefficients will be useful:
\begin{itemize}
\item If $F,G$ are polynomials in $\Z[Y_1,\dots,Y_s]$ such that $G$
  divides $F$ in $\Z[Y_1,\dots,Y_s]$, we get
  from~\cite[Proposition~2.12]{Lang83} that
\begin{equation}\label{eq:lang}
  h(G) \le h(F) + s\deg(F).
\end{equation}
\item If $S$ is an $N \times N$ matrix with entries in
  $\Z[Y_1,\dots,Y_s]$ of column degrees at most $d_1,\dots,d_N$ and
  column heights at most $h_1,\dots,h_N$, its determinant (if nonzero) has height
  \begin{equation}\label{eq:det}
    h(\det(S)) \le N\ln(N) + \sum_{i \le N} h_i + \ln(s+1) \sum_{i \le N} d_i,
  \end{equation}
  see~\cite[Lemma~1.2]{KrPaSo01}.
\end{itemize}

\subsection{Basics on reduction modulo $p$}

One of the recurrent arguments we will make is that the result of
certain operations made over $\Q$ ``reduce well'' modulo a prime $p$.
The following notation will be handy in such circumstances.

\begin{definition}
For $p$ a prime, we denote by $\Z_{(p)}$ the localization of $\Z$ at the prime ideal $(p)$.
\end{definition}
This is the subring of $\Q$ consisting of fractions whose minimal
denominator does not vanish modulo $p$.  There is a well-defined
reduction modulo $p$ homomorphism $\Z_{(p)} \to \F_p$, written $F
\mapsto f=F \bmod p$; the notation carries over to polynomials with
coefficients in $\Z_{(p)}$.

In the proof of the following lemma, and of a few others, we will
rewrite a nonzero polynomial $A$ with rational coefficients as $A =
n_A A'/d_A$, with $n_A,d_A$ positive coprime integers and $A'$
primitive, that is, with integer coefficients and of contents $1$.
Gauss' lemma states that the product of primitive polynomials is
primitive: this is well-known in one variable, and remains true for
multivariate polynomials (use Kronecker substitution with sufficiently
high exponents to make sure it is one-to-one on the exponents of all
polynomials we consider).

\begin{lemma}\label{lemma:divmodp}
  Let $A,B,C$ be polynomials with coefficients in $\Q$, such that $A =
  BC$.  Let $p$ be a prime and suppose that
  \begin{itemize}
  \item $A$ and $B$ have coefficients in $\Z_{(p)}$, so that we can define
    $a=A \bmod p$ and $b=B \bmod p$,
  \item $B$ does not vanish modulo $p$.
  \end{itemize}
  Then $C$ has coefficients in $\Z_{(p)}$, and if we write $c =C \bmod p$, we have $a = bc$.
\end{lemma}

\begin{proof}
  Let us first prove that $C$ has coefficients in $\Z_{(p)}$.  Write
  $A=n_A A'/d_A$, $B=n_B B'/d_B$ and $C=n_C C'/d_C$, with $A',B',C'$
  primitive and $(n_A,d_A)$, $(n_B,d_B)$, $(n_C,d_C)$ pairs of coprime
  integers, with in particular $d_A$ being the lcm of the denominators
  of the coefficients in $A$ (and the same for $B$ and $C$).

  Using Gauss' lemma, we get $n_A d_B d_C = d_A n_B n_C$. If $p$ divides
  $d_C$, it must then divide $d_A n_B$ (because $d_C$ and $n_C$ are coprime). This is impossible:
  \begin{itemize}
  \item because $A$ has coefficients in $\Z_{(p)}$, $d_A \bmod p$ is nonzero
  \item because $b$ is nonzero, $n_B \bmod p$ is nonzero.
  \end{itemize}
  As a result, $p$ does not divide $d_C$. Now that we
  know that $C$ has coefficients in $\Z_{(p)}$, we can reduce $A=BC$
  modulo $p$, and we are done.
\end{proof}


\subsection{GCDs and their reduction modulo $p$}

GCDs are not unique in multivariate polynomial rings. Suppose that
$\K$ is a field and consider polynomials $F$ and $G$ in
$\K[Y_1,\dots,Y_s]$. The GCDs of $F$ and $G$ are all those polynomials
whose irreducible factors are the common irreducible factors of $F$
and $G$, raised to the minimum of their exponents in the
decompositions of $F$, resp. $G$. They all differ by nonzero multiplicative
constants. We will use the following convention to normalize them.

\begin{definition}\label{def:prec}
  Let $\prec$ be the graded lexicographic ordering induced by $Y_1 >
  \cdots > Y_s$ on $\K[Y_1,\dots,Y_s]$. For $F,G$ not both zero, 
  the {\em monic GCD} of $F$ and $G$, written $\gcd(F,G)$, is the
  unique GCD of these polynomials with leading coefficient $1$ for
  the ordering $\prec$.
\end{definition}

\begin{remark}
  If $\L$ is a field extension of $\K$, $\gcd(F,G)$ is the same whether
  taken in $\K[Y_1,\dots,Y_s]$ or $\L[Y_1,\dots,Y_s]$.
\end{remark}

\medskip

In all that follows, when we mention the leading coefficient of a
polynomial $F$, we refer to the ordering introduced above; we denote
it by $\lc(F)$. Leading coefficients depend on the way we order
variables, so that the monic GCD of two polynomials $F,G$ in
$\K[Y_1,\dots,Y_s]$ will in general differ from the monic GCD of $F,G$ seen in
$\K[Y_{\sigma(1)},\dots,Y_{\sigma(s)}]$, for a permutation $\sigma$ of
$\{1,\dots,s\}$.

\begin{lemma}\label{lemma:gcd2}
  Let $F,G$ be nonzero in $\Q[Y_1,\dots,Y_s]$, with respective degrees
  at most $d_F,d_G$ and heights $h_F,h_G$, and let $H$ be their monic
  GCD. There exists a nonzero integer $\Delta_{F,G}$ with
   \[
  h(\Delta_{F,G}) \in O(
  s d_F \ln(d_F) + s d_G \ln(d_G) + s d_F h_G  + s d_Gh_F + s^2 d_Fd_G)
  \] such that if
  a prime $p$ does not divide $\Delta_{F,G}$, $F,G,H$ have
  coefficients in $\Z_{(p)}$, and if we denote by
  $f,g,h$ their reductions modulo $p$, $h = \gcd(f,g)$.
\end{lemma}
\begin{proof}
  Write $F=n_F F'/d_F$, $G=n_G G'/d_G$ and $H= n_H H'/d_H$, with
  $F',G',H'$ primitive in $\Z[Y_1,\dots,Y_s]$ and $(n_F,d_F)$,
  $(n_G,d_G)$, $(n_H,d_H)$ pairs of coprime integers. Since $H$ is
  monic, $\lc(d_H H) = d_H$, but this is equal to $n_H \lc(H')$, so $n_H$
  divides $d_H$; this in turn implies that $n_H=1$.

  Define also $A=F'/H'$ and $B=G'/H'$; these are polynomials with {\it
    a priori} rational coefficients. From the equality $F'=A H'$, if
  we write as usual $A =n_A A'/d_A$, we get $d_A F'=n_A A' H'$, and by
  Gauss' lemma we obtain $n_A=d_A$ and $A=A'$. It follows that $A$ and
  similarly $B$ are primitive polynomials in $\Z[Y_1,\dots,Y_s]$, with
  no nonconstant common factor in $\Q[Y_1,\dots,Y_s]$.
  
  Our first condition is that $p$ should not divide $\delta = d_F d_G
  \ell_F$, where $\ell_F$ is the leading coefficient of $d_F F$, so
  that $f=F\bmod p$ and $g =G \bmod p$ are well-defined. The equality
  $F' = A H'$ gives $d_F F = n_F A (d_H H)$, so the leading
  coefficient of $d_H H$ divides that of $d_F F$ in $\Z$, that is,
  $d_H$ divides $\ell_F$. Thus, our assumption also implies that $d_H$
  does not vanish modulo $p$, allowing us to define $h = H \bmod p$.

  We can then reduce the relation $d_F F = n_F d_H A H$ modulo $p$;
  since we know that $d_F$ is a unit modulo $p$, we get that $h$
  divides $f$. Similarly, $h$ divides $g$.  Because $h$ is monic, this
  shows that $\gcd(f,g) = h \gcd(a,b)$, so it is enough to give
  conditions on $p$ that guarantee that $a$ and $b$ remain coprime
  modulo $p$.

  For $i=1,\dots,s$, let $D_i \in \Z[Y_1,\dots,\hat Y_i,\dots,Y_s]$ be
  the resultant of $A$ and $B$ with respect to $Y_i$. $D_i$ is
  nonzero: otherwise, $A$ and $B$ would have a common factor in
  $\Q(Y_1,\dots,\hat Y_i,\dots,Y_s)[Y_i]$ of positive $Y_i$-degree,
  which would result in the existence of a nonconstant common factor
  for $A$ and $B$ in $\Q[Y_1,\dots,Y_s]$ (proof: write $A= PQ$ in
  $\Q(Y_1,\dots,\hat Y_i,\dots,Y_s)[Y_i]$, with $P$ of positive
  $Y_i$-degree; rewrite as $n_A A' = n/d P'Q'$ with $A'$, $P'$ and
  $Q'$ primitive in $\Q[Y_1,\dots,\hat Y_i,\dots,Y_s][Y_i]$; deduce
  that $n_A=n/d$ and $A'=P'Q'$, multiply by $n_A$ to get $A=n_A P'Q'$,
  so $P'$ divides $A$ in $\Q[Y_1,\dots,Y_s]$. Similarly, it divides $B$).

  Suppose that $p$ is such that $D_i \bmod p$ is not zero. It follows
  that the leading $Y_i$-coefficients of $A$ and $B$ do not both
  vanish modulo $p$ (otherwise, the first row of their $Y_i$-Sylvester
  matrix vanishes). Following the proof of Lemma~6.25
  in~\cite{GaGe13}, we see that the GCD of $a=A \bmod p$ and $b=B \bmod p$
  in $\F_p(Y_1,\dots,\hat Y_i,\dots,Y_s)[Y_i]$ is constant.  In
  particular, $a$ and $b$ do not have any common factor
  in $\F_p[Y_1,\dots,Y_s]$ of positive $Y_i$-degree.

  Repeat this for all $i$; we arrive at the conclusion that if all
  $D_i \bmod p$ are nonzero, $a=A \bmod p$ and $b=B \bmod p$ have no
  nonconstant common factor in $\F_p[Y_1,\dots,Y_s]$, which is the
  conclusion we sought to guarantee. For $i\le s$, let then $\delta_i \in
  \Z$ be an arbitrary nonzero coefficient of $D_i$, and let
  $\Delta_{F,G} =\delta \delta_1 \cdots \delta_s$.
  
  To conclude, it is enough to give an upper bound on the height of
  the polynomials $D_i$.  For $i \le s$, $A$ and $B$ have
  $Y_i$-degrees at most $d_F$, resp. $d_G$, and height at most $h_F +
  s d_F$, resp. $h_G + s d_G$ (see remarks in
  Subsection~\ref{sec:defH}). It follows that $D_i$, and thus
  $\delta_i$, have height at most
  \[ (d_F+d_G)\ln(d_F+d_G) + d_F(h_G+sd_G) + d_G(h_F + sd_F) + 2\ln(s+1) d_F d_G.\]
  Multiply by $s$ to take all variables into account, and add
  $2h_F+h_G$ for $\delta$. After minor simplifications, this gives the bound
  claimed in the lemma.
\end{proof}


\section{The Chow forms}\label{sec:chow}

We review the definition of the Chow form of an algebraic set, and how
it allows us to define its height, when we are working over $\Q$. Again,
our main reference is~\cite{KrPaSo01}.

A few basic concepts are taken for granted: for a field $\K$ and its
algebraic closure $\Kbar$, an {\em algebraic set} $V \subset \Kbar^n$
is the set of common zeros of some polynomials $(F_i)$ in
$\Kbar[X_1,\dots,X_n]$.  An algebraic set $V$ can uniquely be written
as the union of its irreducible components, which are the zero-sets of
the prime ideals that appear in the primary (here, prime) decomposition of $I(V)$.
For $r=0,\dots,n$, the {\em $r$-equidimensional component} of $V$, which we
will denote on occasion $V_r$, is the union of its irreducible
components of dimension $r$. Finally, we will say that $V$ is $r$-{\em
  equidimensional} if it is not empty and $V=V_r$.

One defines dimension and degree, first for irreducible algebraic
sets, then for arbitrary ones (by taking maximum, resp. sum over
irreducible components).

An algebraic set $V$ is {\em defined over $\K$} if its defining ideal
$I(V) \subset \Kbar[X_1,\dots,X_n]$ admits generators with
coefficients in $\K$. If it is the case, even though the irreducible
components of $V$ may not necessarily be defined over $\K$, we can
define its {\em $\K$-irreducible components}, which are obtained by
taking the prime ideals in the primary decomposition of $I(V) \cap
\K[X_1,\dots,X_n]$. They are all defined over $\K$, and the
$r$-equidimensional component of $V$ is then equal to the union of the
$\K$-irreducible components of $V$ of dimension $r$.



\subsection{Defining the Chow forms}\label{sec:defchow}

Let $\K$ be a field and let $V \subset \Kbar^n$ be an
$r$-equidimensional algebraic set of degree $D$. Consider the set
$\Delta_V \subset \P^n \times \cdots \times \P^n$ ($r+1$ factors) of
all $\bm u=(\bm u_0,\dots,\bm u_r)$, with $\bm u_i=(u_{i,0} : \cdots :
u_{i,n})$ for all $i$, such that there exists $(x_1,\dots,x_n)$ in $V$
for which
\begin{equation}\label{eq:chow}
 u_{i,0} + u_{i,1} x_1 + \cdots + u_{i,n} x_n = 0, \quad i=0,\dots,r.  
\end{equation}
The Zariski closure of $\Delta_V$ is a hypersurface, and a {\em Chow
  form} of $V$ is any squarefree polynomial that defines this
hypersurface. All these polynomials differ by a nonzero multiplicative
factor in $\Kbar$; they lie in $\Kbar[\bm U_0,\dots,\bm U_r]$,
where $\bm U_0,\dots,\bm U_r$ are new indeterminates, with $\bm U_i =
U_{i,0},\dots,U_{i,n}$ for all $i$,  are homogeneous of degree $D$
in each block of variables $\bm U_i$ and have no nontrivial multiple
factors in $\Kbar[\bm U_0,\dots,\bm U_r]$. If $V$ is defined over
$\K$, it admits Chow forms with coefficients in $\K$ as well. The
following properties are all classical.


\begin{fact}
  Let $V^{\rm proj}$ be the projective closure of $V$ in $\P^n$, and
  let $C$ be a Chow form of $V$.  Then, $\bm u=(\bm u_0,\dots,\bm u_r)
  \in \P^n \times \cdots \times \P^n$ cancels $C$ if and only if there
  exists $(x_0 : \cdots : x_n)$ in $V^{\rm proj}$ such that
  $u_{i,0}x_0 + u_{i,1} x_1 + \cdots + u_{i,n} x_n = 0$ for all $i$.
\end{fact}
See Lemma 1.1 in~\cite{JeKrSaSo04} for a proof.

\begin{fact}\label{fact:unionchow}
  The Chow forms of an equidimensional variety are the products of the
  Chow forms of its irreducible components, and the latter are
  irreducible in $\Kbar[\bm U_0,\dots,\bm U_r]$.
\end{fact}

\begin{fact}
  The Chow forms of $\A^n$ and $\P^n$ are $c \det(\bm U_0,\dots,\bm U_n)$,
  $c \in \Kbar-\{0\}$. 
\end{fact}

\medskip

One can extend the definition of the Chow form, by considering
intersections of $r+1$ forms of arbitrarily chosen degrees rather than
just hyperplanes, as in~\cite{Philippon86}. In this paper, we will use
the following construction, called generalized Chow forms
in~\cite{JeKrSaSo04,KrPaSo01}: we work in the affine setting, and we
replace the last linear form $U_{r,0} + U_{r,1} X_1 + \cdots +
U_{r,n}X_n$ by a polynomial $F(X_1,\dots,X_n)$ of degree $d$ with
generic coefficients $\bm V=(V_{\alpha})$, for $\alpha \in
\N^n$ such that $|\alpha| \le d$ (the other linear forms do not
change). By the same construction as in the case $d=1$, we now obtain
a set $\Delta^{(d)}_V$ in $\P^{n}\times \cdots \times \P^n\times
\P^{N-1}$, with $N={{n+d} \choose n}$. For $d=1$, we recover the
previous definition, up to replacing the $n+1$ indeterminates $\bm V$
by $\bm U_{r}=U_{r,0},\dots,U_{r,n}$.

Any defining polynomial of the Zariski closure of $\Delta^{(d)}_V$ is
called a {\em degree $d$ Chow form} of $V$. These are multi-homogeneous
polynomials in variables $\bm U_0,\dots,\bm U_{r-1},\bm V$, with
degree $D$ in $\bm V$ and $dD$ in all other groups of variables; they
all differ by a nonzero multiplicative constant factor. If $C_d \in
\Kbar[\bm U_0,\dots,\bm U_{r-1},\bm V]$ is such a polynomial, and $F$ is
in $\Kbar[X_1,\dots,X_n]$ of degree at most $d$, we will denote by
$C_d[F]$ the polynomial $C_{d}(\bm U_0,\dots,\bm U_{r-1},\acoeff(F))$,
where $\acoeff(F)$ is the vector of coefficients of $F$.


\subsection{Normalizing Chow forms}

We now discuss how to define a canonical choice for the Chow form of
an algebraic set $V$.  When $V$ is defined over $\Q$, we did mention
its primitive Chow forms in the introduction, but their definition
does not carry over to arbitrary fields. Instead, we will rely on a
construction introduced by Krick, Pardo and Sombra~\cite{KrPaSo01},
which is applicable for $V$ in generic coordinates.

In all that follows, in a ring such as $\Kbar[\bm U_0,\dots,\bm U_r]$, we
use the graded order introduced in Definition~\ref{def:prec}, which is induced
by $U_{0,0} > \dots > U_{0,n} > \cdots > U_{r,0} > \cdots > U_{r,n}$.

\begin{lemma}\label{lemma:leadmon}
  If $V$ is an $r$-equidimensional algebraic set of degree $D$ and $C$
  is a Chow form of it, $C$ has no monomial greater than $U_{0,0}^{D}
  \cdots U_{r-1,r-1}^{D} U_{r,r}^{D}$ .
\end{lemma}
\begin{proof}
   A Chow form of $V$ is homogeneous of degree $D=\deg(V)$ in the
   $(r+1)$-minors $M_{i_0,\dots,i_r}$ of the $(r+1)\times (n+1)$
   matrix with entries $U_{i,j}$~\cite{DaSt95}, and the leading
   coefficient of each such minor is less than or equal to $U_{0,0}
   \cdots U_{r-1,r-1}U_{r,r}$.
\end{proof}

\begin{definition}
  An $r$-equidimensional algebraic set $V$ of degree $D$
  is in {\em normal position} if the leading monomial of its Chow
  forms is $U_{0,0}^{D} \cdots U_{r-1,r-1}^{D} U_{r,r}^{D}$.

  If this is the case, we denote by $C_V$ the unique Chow form of $V$ with
  leading coefficient $1$, and we call it the {\em normalized Chow form}
  of $V$.
\end{definition}
  

\begin{lemma}\label{lemma:normal}
  If $V_1,\dots,V_s$ are pairwise distinct irreducible and
  $r$-dimensional algebraic sets, and
  $V=V_1 \cup \cdots\cup V_s$, then:
  \begin{itemize}
  \item $V$ is in normal position if and only if all $V_i$'s are
  \item if $V$ is in normal position, then $C_V = C_{V_1} \cdots C_{V_s}$.
  \end{itemize}
\end{lemma}
\begin{proof}
  Let $C_1,\dots,C_s$ be Chow forms of $V_1,\dots,V_s$, so that $C=C_1
  \cdots C_s$ is a Chow form of $V$ (Fact~\ref{fact:unionchow}). If we
  write $D_i =\deg(C_i)$ for all $i$ and $D=D_1 +\cdots+D_s$,
  Lemma~\ref{lemma:leadmon} shows that all monomials in $C_i$ are less
  than or equal to $U_{0,0}^{D_i} \cdots U_{r,r}^{D_i}$. This implies
  that the coefficient $\ell$ of $U_{0,0}^{D} \cdots U_{r,r}^{D}$ in
  $C$ is the product of the coefficients $\ell_i$ of the corresponding
  monomials $U_{0,0}^{D_i} \cdots U_{r,r}^{D_i}$ in
  $C_1,\dots,C_s$. In particular, $\ell$ is nonzero if and only if all
  $\ell_i$'s are (proving the first point), and is equal to one when all
  $\ell_i$'s are (proving the second point).  
\end{proof}

In~\cite{KrPaSo01}, Krick {\it et al.} gave a sufficient condition for
$V$ to be in normal position: this is the implication $(2) \implies
(3)$ of the following lemma. Below, a {\em generic} property is one
that holds in a non-empty Zariski-open set of the corresponding
parameter space.
\begin{lemma}\label{lemma:normaleq}
  Let $V \subset \Kbar^n$ be $r$-equidimensional of degree $D$. The following
  are equivalent:
  \begin{enumerate}
  \item for generic $x_1,\dots,x_r$ in $\Kbar$, the set $V \cap V(X_1-x_1,\dots,X_r-x_r)$ has cardinality $D$
  \item there exist $x_1,\dots,x_r$ in $\Kbar$ such that the set $V \cap V(X_1-x_1,\dots,X_r-x_r)$ has cardinality $D$
  \item $V$ is in normal position
  \item the projective algebraic set $V^{\rm proj} \cap V(X_0,\dots,X_r)$ is empty.
  \end{enumerate}
\end{lemma}
\begin{proof}
  Obviously, (1) implies (2). Conversely, assume (2), and let us prove
  (1). We first prove that we can assume $V$ irreducible.

  For $x_1,\dots,x_r$ as in (2), and for any irreducible component $W$ of
  $V$, with degree $D_W$, $W \cap V(X_1-x_1,\dots,X_r-x_r)$ is finite
  (because the fibre of $V$ is), and thus of cardinality $d_W \le D_W$
  (by B\'ezout). But since $V \cap V(X_1-x_1,\dots,X_r-x_r)$ has
  cardinality $D =\sum D_W$, we conclude $d_W=D_W$ for all $W$.  If we
  prove (1) for all $W$, then (1) follows for $V$ as well, since for
  generic $x_1,\dots,x_r$, the fibres $W \cap V(X_1-x_1,\dots,X_r-x_r)$
  are pairwise disjoint.

  Assuming here that $V$ is irreducible, let us then prove (1), and
  let us write $\pi_r$ for the projection on the
  $X_1,\dots,X_r$-space. If the Zariski closure of $\pi_r(V)$ has
  dimension less than $r$, by the theorem on the dimension of
  fibres~\cite[Theorem~9.9]{MilneAG}, all nonempty fibres of the
  restriction of $\pi_r$ to $V$ have positive dimension, which
  contradicts (2). It follows that $\pi_r: V\to \Kbar^r$ is dominant, with
  generically finite fibres. Then, by~\cite[Proposition~1]{Heintz83},
  all finite fibres of this mapping have cardinality less than or
  equal to the cardinality of the generic fibre, so the generic
  fibre has degree at least $D$. On the other hand, it cannot have
  degree more than $D$, so we are done.

  Next, we prove that (3) is equivalent to (4). Let $V^{\rm proj}$ be
  the projective closure of $V$ and let $C$ be a Chow form of $V$. If
  we let $D=\deg(V)$, the coefficient of $U_{0,0}^{D} \cdots
  U_{r,r}^{D}$ in $C$ is obtained by evaluating $C$ at $\bm U_i=
  (0,\dots,0,1,0,\dots,0)$ for all $i$ with the value $1$ at entry
  $U_{i,i}$. In particular, this coefficient is nonzero (that is, $V$
  is in normal position) if and only if there exists no $(x_0 : \cdots
  : x_n)$ in $V^{\rm proj}$ such that $x_0 = \cdots = x_r = 0$, that
  is, $V^{\rm proj} \cap V(X_0,\dots,X_r)$ is empty. This proves our claim.
  
  Finally, we prove that (1) and (2) are both equivalent to (4). 
  Suppose that (2) holds, and consider the intersection $V^{\rm proj} \cap
  V(X_1-x_1 X_0,\dots,X_r-x_r X_0)$. By B\'ezout's inequality (for
  projective varieties), it has degree at most $D$, but our assumption
  shows that it contains $D$ isolated points (at finite distance), so
  it has no point at infinity. This shows that
  \[ V^{\rm proj} \cap  V(X_0,X_1-x_1 X_0,\dots,X_r-x_r X_0) = V^{\rm proj} \cap  V(X_0,X_1,\dots,X_r)\]
  is empty, so we get (4).

  Finally, suppose that (4) holds, so that $V^{\rm proj} \cap
  V(X_0,X_1,\dots,X_r)$ is empty. By~\cite[Corollary~5.6]{Mumford76},
  the degree of the restriction of $\pi_r$ to $V^{\rm proj}$ is equal
  to $D=\deg(V)$. The degree of such a mapping is by definition the
  cardinality of its generic fibre, so (1) holds.
\end{proof}

Suppose that $V$ is in normal position, so that $V^{\rm proj} \cap
V(X_0,\dots,X_{r-1},X_r)$ is empty. It follows that for $d \ge 1$,
$V^{\rm proj} \cap V(X_0,\dots,X_{r-1},X^d_r)$ is empty. In turns, this implies
that if $C$ is a degree $d$ Chow form of $V$,
$C({\bm e}_0,\dots,{\bm e}_{r-1},{\bm e}_{r,d})\ne 0$, where ${\bm e}_i$ is the vector of
coefficients of $X_i$ (for $0 \le i < r$), and ${\bm e}_{r,d}$ is the vector
of coefficients of $X_r^d$.

\begin{definition}
  If $V$ is in normal position, for $d \ge 1$ we denote by $C_{d,V}$
  the unique degree $d$ Chow form of $V$ such that
  $C_{d,V}({\bm e}_0,\dots,{\bm e}_{r-1},{\bm e}_{r,d})=1$.
\end{definition}


\subsection{Characteristic polynomials}

The last family of polynomials we will need in our algorithms are {\em
  characteristic polynomials} of an algebraic set. Let $V \subset
\Kbar^n$ be an $r$-equidimensional algebraic set, with $\Kbar[\bm
  U_0,\dots,\bm U_r]$  the ring where its Chow forms lie.  Consider
new blocks of variables $\bm U'_0,\dots,\bm U'_r$, with $\bm U'_i =
T_i,U_{i,0},\dots,U_{i,n}$ ($T_i$ a new symbol) for all $i$. Define
also $\bm v_i = U_{i,0}-T_i,U_{i,1},\dots,U_{i,n}$.

\begin{definition}
  Let $V$ be an $r$-equidimensional algebraic set of degree $D$.  The
  {\em characteristic polynomials} of $V$ are all polynomials of the
  form $P=(-1)^D C(\bm v_0,\dots,\bm v_r) \in \Kbar[\bm U'_0,\dots,\bm
    U'_r]$, for $C$ a Chow form of $V$.
\end{definition}
The characteristic polynomials of $V$ are homogeneous of degree
$D=\deg(V)$ in the groups of variables $\bm U'_i$ for all $i$, so
their total degree is $rD$. See~\cite[Section~3.1]{JeKrSaSo04} for a
geometric motivation for this construction.

For the next lemma, recall that in $\K[\bm U'_0,\dots,\bm U'_r]$, we
use the graded lexicographic order induced by $T_0 > U_{0,0} >
\cdots > U_{0,n} > \cdots > T_r > U_{r,0} > \cdots > U_{r,n}$.
\begin{lemma}\label{lemma:defcharpoly}
  Let $V$ be an $r$-equidimensional algebraic set of degree $D$ and in
  normal position, and let $C_V$ be its normalized Chow form.  Then,
  $P_V=(-1)^D C_V(\bm v_0,\dots,\bm v_r)$ has leading monomial $T_0^D
  U_{1,1}^D \cdots U_{r,r}^D$ and leading coefficient 1.
\end{lemma}
\begin{proof}
  Write $M_0=T_0^D U_{1,1}^D \cdots U_{r,r}^D$. Consider a monomial
  $m=\bm U_0^{\bm \alpha_0} \cdots \bm U_r^{\bm \alpha_r}$ in a Chow
  form $C$ of $V$, and let $M$ be its evaluation at $\bm v_0,\dots,\bm
  v_r$. For our monomial order, up to sign, the leading monomial in
  $M$ is $m'$, which is obtained by replacing $U_{i,0}$ by $T_i$ in
  $m$ (for all $i$). We know that $m$ is less than or equal to $U_{0,0}^D \cdots
  U_{r,r}^D$ (Lemma~\ref{lemma:leadmon}), and as a result, $m'$ is
  less than or equal to $M_0$.

  Since $V$ is in normal position, its normalized Chow form $C_V$
  features the monomial $U_{0,0}^D \cdots U_{r,r}^D$ with coefficient
  $1$; its evaluation at $\bm v_0,\dots,\bm v_r$ then has the monomial
  $M_0$ above, with coefficient $(-1)^D$.
\end{proof}
Naturally, we call $P_V$ as above the {\em normalized} characteristic
polynomial of $V$. The next lemma follows directly from
Lemma~\ref{lemma:normal}.

\begin{lemma}\label{lemma:normalcharpoly}
  If $V_1,\dots,V_s$ are pairwise distinct irreducible $r$-dimensional
  algebraic sets, all in normal position, and $V=V_1 \cup \cdots\cup
  V_s$, then $P_V = P_{V_1} \cdots P_{V_s}$.
\end{lemma}


\subsection{Height of varieties}

Finally, we introduce the notion of height of varieties we use in this
paper, following~\cite{Philippon95,KrPaSo01}, and we state its main
properties.

Let $V \subset \C^n$ be $r$-equidimensional and defined over $\Q$ and
let $C$ be any Chow form of $V$ with coefficients in $\Q$. Then, we
set
\[h(V) = \sum_{p~ {\rm prime}} \ln(|C|_p) + m(C, S^{r+1}_{n+1}) + (r+1) \deg(V) \sum_{i=1}^n \frac 1{2i},\]
where
\begin{itemize}
\item $|C|_p=\max_\alpha (|c_\alpha|_p)$, for all nonzero coefficients
  $c_\alpha$ of $C$, where $|\ |_p$ is the $p$-adic absolute value,
  given by $|p|_p=1/p$ and $|p'|_p=1$ ($p'$ prime $\ne p$), and extended by
  multiplicativity to $\Q-\{0\}$,
\item $m(C, S^{r+1}_{n+1})$ is a variant of the usual Mahler measure,
  defined as the integral of $\ln(|C|)$ over
  $S_{n+1}\times\cdots\times S_{n+1}$ ($r+1$ factors), where $S_{n+1}$
  is the unit sphere in $\C^{n+1}$, endowed with the Haar measure of
  mass $1$.
\end{itemize}
This does not depend on the choice of $C$, as long as it has rational
coefficients, and the height of an equidimensional algebraic set $V$
defined over $\Q$ is the sum of the heights of its $\Q$-irreducible
components. Finally, for an arbitrary $V$, still defined over $\Q$, we
let $h(V)$ be the sum of the heights of its equidimensional
components.

If $V$ is in normal position, we highlighted a distinguished choice
for its Chow form, namely the normalized Chow form $C_V$, but another
possibility was pointed out in the introduction: choosing a primitive
Chow form (thus, with integer coefficients) $C_V^{\rm primitive}$,
which is unique up to sign.  

\begin{lemma}\label{lemma:hCV}
    When $V$ is in normal position, the equality $h(C_V)=h(C_V^{\rm
      primitive})$ holds.
\end{lemma}
\begin{proof}
  Because $C_V$ has leading coefficient $1$, we can
  write $C_V = C_V^{\rm primitive}/d$, where $d$ is the leading
  coefficient of $C_V^{\rm primitive}$. It follows from
  Definition~\ref{def:height} that $h(C_V)=h(C_V^{\rm primitive})$.
\end{proof}


\begin{lemma}\label{lemma:hchow}
  When $V$ is in normal position, we have
  \[ h(C_V) \le h(V) + (r+1)\ln(n+2)\deg(V).\]
\end{lemma} 
\begin{proof}
 Using $C_V^{\rm primitive}$ in the
  expression for $h(V)$, all $p$-adic terms in the definition of
  $h(V)$ vanish; we are then left with
  \[h(V) =  m(C_V^{\rm primitive}, S^{r+1}_{n+1}) + (r+1) \deg(V) \sum_{i=1}^n \frac 1{2i}.\]
  Let $m(C_V^{\rm primitive})$ be the Mahler measure of the polynomial $C_V^{\rm primitive}$, that is, the integral of
  $\ln(|C_V^{\rm primitive}|)$ over the torus $S_1^{(r+1)(n+1)}$. We have from~\cite[Eq. (1.2)]{KrPaSo01}
  \begin{equation}\label{eq:ineqmm}
   m(C_V^{\rm primitive}) \le m(C_V^{\rm primitive}, S^{r+1}_{n+1}) + (r+1) \deg(V) \sum_{i=1}^n \frac 1{2i} = h(V).    
  \end{equation}
  We can then apply~\cite[Lemma~1.1]{KrPaSo01}, which gives
  \[ h(C_V^{\rm primitive}) \le m(C_V^{\rm primitive}) + (r+1)\ln(n+2)\deg(V) =  h(V) + (r+1)\ln(n+2)\deg(V),\]
  and the previous lemma allows us to conclude.
\end{proof}

To prove a similar estimate on degree $d$ Chow forms, we will use the
{\em local heights} of an algebraic set $V$, which are defined in
Section~1.2.4 of~\cite{KrPaSo01}. For $V$ defined over $\Q$, $r$-equidimensional and in
normal position, set
\begin{itemize}
\item $h_\infty(V)=m(C_V, S_{n+1}^{r+1})+ (r+1) \deg(V) \sum_{i=1}^n \frac 1{2i}$
\item $h_p(V)=\ln(|C_V|_p)$, $p$ prime.
\end{itemize}
In particular, we see that $h(V)=h_\infty(V) + \sum_p h_p(V)$.
Because $C_V$ has a coefficient equal to $1$, we see that $h_p(V) \ge
0$ for all $p$. The same holds for the height at infinity; this claim is
most likely well-known, but we did not find a reference for it.

\begin{lemma}
  For $V$ in normal position, $h_\infty(V)\ge 0$.
\end{lemma}
\begin{proof}
  As in we did in Equation~\eqref{eq:ineqmm} in the proof of the previous lemma (but
  applied to $C_V$ this time), we use~\cite[Eq. (1.2)]{KrPaSo01} to get
  \[ 
  m(C_V) \le m(C_V, S^{r+1}_{n+1}) + (r+1) \deg(V) \sum_{i=1}^n \frac 1{2i} = h_\infty(V).
  \]
  We are then left with proving that $m(C_V)$ itself is non-negative.

  For a nonzero polynomial $P \in \C[Y_1,\dots,Y_t]$, we have the
  inequality between Mahler measures $m(P) \ge m(P_{d_1})$, if we let
  $P_{d_1} \in \C[Y_2,\dots,Y_t]$ be its leading coefficient with
  respect to $Y_1$ (see~\cite[p.~117]{Boyd1981}). Taking repeated
  leading coefficients, we deduce that $m(P) \ge \ln(|\lc_{\rm
    lex}(P)|)$, where $\lc_{\rm lex}(P)$ is the leading
  coefficient of $P$ for the lexicographic order $Y_1 > \cdots > Y_t$.

  We saw that the leading monomial of the normalized Chow form $C_V$
  for the graded lexicographic order $U_{0,0} > \cdots > U_{r,n}$ is
  $U_{0,0}^{D} \cdots U_{r-1,r-1}^{D} U_{r,r}^{D}$, for $D=\deg(V)$ and $r=\dim(V)$,
  with leading coefficient $1$. Since $C_V$ is homogeneous, it admits
  the same leading monomial for the lexicographic order, and the
  conclusion follows.
\end{proof}


The previous claim allows us to derive an upper bound on the height of degree $d$
Chow forms. It will be enough to bound the height of $C_{d,V}[F]$ for some
given $F$ (see the definition at the end of Subsection~\ref{sec:defchow}).
\begin{lemma}\label{lemma:hcdvf}
  Suppose that $V$ is defined over $\Q$, $r$-equidimensional and in
  normal position. For $d \ge 1$ and $F$ in $\Z[X_1,\dots,X_n]$ of
  degree at most $d$ and height at most $h$, we have
  \[ h(C_{d,V}[F]) \le
  d h(V) + h \deg(V) + (r+1)\ln(n+1) d\deg(V).\]
\end{lemma}
\begin{proof}
  Lemma~2.1 in~\cite{KrPaSo01} shows the following inequalities:
  \begin{itemize}
  \item $m(C_{d,V}[F], S_{n+1}^r) + r d \deg(V) \sum_{i=1}^n 1/2i \le d
    h_\infty(V) + h \deg(V) + \ln(n+1) d\deg(V)$
  \item $h_p(C_{d,V}[F]) \le d h_p(V) $ if $p$ is prime,
  \end{itemize}
  where $h_p(C_{d,V}[F])=\max(0,\ln(|C_{d,V}[F]|_p))$ for $p$ prime.
  The assumptions in~\cite{KrPaSo01} differ from ours in a minor way:
  that reference assumes that $V \cap V(X_1,\dots,X_r)$ has
  cardinality $\deg(V)$. However, that assumption is only used to prove (as
  in Lemma~\ref{lemma:normaleq}) that $V$ is in normal position.

  Because $C_{d,V}[F]$ has degree at most $d \deg(V)$ in each group of
  variables $\bm U_i$, Equation~(1.2) in~\cite{KrPaSo01} shows that
  the Mahler measure of $C_{d,V}[F]$ satisfies
  \[ m(C_{d,V}[F]) \le m(C_{d,V}[F], S^{r}_{n+1}) + r d \deg(V) \sum_{i=1}^n \frac 1{2i}.\]
  As a consequence, we have $m(C_{d,V}[F]) \le d h_\infty(V) + h \deg(V) +
  \ln(n+1) d\deg(V)$, and Lemma~2.12 in~\cite{KrPaSo01} gives
  \begin{equation}\label{ineq:lncdvF}
    \ln(|C_{d,V}[F]|) \le d h_\infty(V) + h \deg(V) + (r+1)\ln(n+1) d\deg(V),    
  \end{equation}
  where $|C_{d,V}[F]|$ is the maximum of the absolute values
  of the coefficients of $C_{d,V}[F]$. Let now $h_\infty(C_{d,V}[F]) =
  \max(0,\ln(|C_{d,V}[F]|))$, and let us establish
  \begin{equation*}\label{ineq:hcdvF}
     h_\infty(C_{d,V}[F]) \le d h_\infty(V) + h \deg(V) + (r+1)\ln(n+1) d\deg(V)
  \end{equation*}
  \begin{itemize}
  \item if $\ln(|C_{d,V}[F]|) \le 0$, so that $h_\infty(C_{d,V}[F]) =
    0$, our inequality follows from the fact that $h_\infty(V) \ge 0$,
    by the previous lemma, so the whole right-hand size is non-negative
  \item else, the claim comes directly from Eq.~\eqref{ineq:lncdvF}.
  \end{itemize}
  To finish, we use the observation~\cite[p.~10]{KrPaSo01} that the
  height of $C_{d,V}[F]$ is given by
  $h(C_{d,V}[F])=h_\infty(C_{d,V}[F]) + \sum_p h_p(C_{d,V}[F])$. Summing
  the upper bounds for $h_p(C_{d,V}[F])$, $p$ prime, and $h_\infty(C_{d,V}[F])$,
  we get $h(C_{d,V}[F]) \le d h(V) + h \deg(V) + (r+1)\ln(n+1) d\deg(V)$.
\end{proof}

  
\section{Basic operations using Chow forms}

In this section, we introduce algorithms to perform basic
set-theoretical operations on algebraic sets through their Chow forms;
the majority of these results are originally found
in~\cite{JeKrSaSo04} and~\cite{BlJeSo04}.

The algorithms typically take multivariate polynomials as input; these
are intended to represent Chow forms of certain algebraic sets,
possibly with the assumption that we are in normal position. Not all
multivariate polynomials represent Chow forms of algebraic sets (see
the discussion of the Chow varieties in~\cite[Chapter 4]{GeKaZe94});
for inputs that do not satisfy such assumptions, the algorithm may
return a (meaningless) output, or {\nondef} in some cases. For
instance, the notation $\gcd( . , .)$ denotes the monic GCD of two
polynomials; it is undefined if both arguments vanish. When this is
the case, our convention is that the procedure that attempted this
computation immediately aborts and returns {\nondef} as well. Other
cases where we will return {\nondef} are when a division by zero
occurs. 

In all cases, we also discuss the special case where the inputs are
defined over $\Q$, and give sufficient conditions on a prime $p$ to
guarantee that all steps in the computation can be carried out modulo
$p$.


\subsection{Union} 

Let $V_1,V_2$ be $r$-equidimensional algebraic sets defined over a
field $\K$, and let $C_1,C_2$ be Chow forms for $V_1,V_2$. By
definition, the Chow forms of $V=V_1 \cup V_2$ are the squarefree
parts of $C_1C_2$ (which are defined up to multiplicative
constants). We will use the following procedure to compute such
polynomials.

\medskip\noindent{\sc Union\_Chow}\vspace{-1mm}
\begin{algorithmic}[1]
\Require polynomials $C_1,C_2$ in $\K[\bm U_0,\dots,\bm U_r]$, with $\bm U_i =U_{i,0},\dots,U_{i,n}$  for all $i$ 
\Ensure  polynomial $C$ in $\K[\bm U_0,\dots,\bm U_r]$ or \nondef
\State $P \gets C_1 C_2$
\State $Q \gets \gcd(P, \partial P/\partial U_{0,0})$
\State \Return $P/Q$
\end{algorithmic}

Note that when either $C_1$ or $C_2$ vanishes (which will never happen
when we use this procedure), $Q$ is undefined (as we only defined
monic GCDs when the arguments are not both zero).

In the proof of
correctness, we will use the assumption that both input polynomials
have trivial GCDs with their partial derivative with respect to
$U_{0,0}$. For algebraic sets in normal position, we first verify that
this assumption is always satisfied over a field of characteristic
zero.

\begin{lemma}\label{lemma:gcdeq1}
  Suppose that $V$ is an $r$-equidimensional algebraic set, defined over a
  field $\K$ of characteristic zero. Then for any Chow form of $V$, $\gcd(C,\partial
  C/\partial U_{0,0})=1$.
\end{lemma}
\begin{proof}
  Suppose that $C$ and $\partial C/\partial U_{0,0}$ admit a
  nonconstant irreducible common factor $A$, say with coefficients in
  $\Kbar$.  If we write $C = A B$, for some polynomial $B$, we know
  that $A$ does not divide $B$ (because $C$ has no multiple factor
  over $\Kbar$). Taking the partial derivative with respect to
  $U_{0,0}$, we see that $A$ divides $\partial A/\partial U_{0,0} B +
  A \partial B/\partial U_{0,0}$, so it divides $\partial A/\partial
  U_{0,0}$. This is possible only if $\partial A/\partial U_{0,0}=0$,
  and because we are in characteristic zero, this implies that $A$
  does not depend on $U_{0,0}$. But $A$ is the Chow form of an
  irreducible component of $V$, so by~\cite[Lemma~3.7]{JeKrSaSo04}, it
  depends on $U_{0,0}$, a contradiction.
\end{proof}  

\begin{lemma}\label{lemma:unionchow}
  Suppose that $V_1$ and $V_2$ are algebraic sets defined over a field
  $\K$, both either empty or $r$-equidimensional and in normal
  position.  Given their normalized Chow forms $C_{V_1},C_{V_2}$, and
  assuming that $\gcd(C_{V_1},\partial C_{V_1}/\partial
  U_{0,0})=\gcd(C_{V_2},\partial C_{V_2}/\partial U_{0,0})=1$ and that
  $\K$ does not have characteristic two, {\sc Union\_Chow} returns the
  normalized Chow form of $V=V_1 \cup V_2$.
\end{lemma}
\begin{proof}
  Write $V = \cup_i W_i \ \cup\ \cup_j Y_j\ \cup\ \cup_k Y'_k$, where
  the $W_i$'s are irreducible components common to $V_1$ and $V_2$,
  the $Y_j$'s are irreducible components of only $V_1$, and the ${Y'_k}$'s
  are irreducible components of only $V_2$. Because we assume that
  $V_1,V_2$ are in normal position, this is also the case for all
  $W_i$, $Y_j$, $Y'_k$ and thus $V$ (first item in
  Lemma~\ref{lemma:normal}), from which we get that $C_V =\prod_i
  C_{W_i} \prod_j C_{Y_j} \prod_k C_{Y'_k}$ (second item in Lemma~\ref{lemma:normal}). On
  the other hand, we also have $C_{V_1} = \prod_i C_{W_i} \prod_j
  C_{Y_j}$ and $C_{V_2} =\prod_i C_{W_i} \prod_k C_{Y'_k}$ (same
  lemma), so $P=C_{V_1} C_{V_2} = \prod_i C_{W_i}^2 \prod_j C_{Y_j} \prod_k C_{Y'_k}$.
  
  Our assumption that $\gcd(C_{V_1},\partial C_{V_1}/\partial
  U_{0,0})=1$ implies that for all $i$, $C_{W_i}$ does not divide
  $\partial C_{W_i}/\partial U_{0,0}$: this would be possible
  only if the latter vanished, but in this case, $C_{W_i}$ would be a
  common factor to $C_{V_1}$ and $\partial C_{V_1}/\partial U_{0,0}$.
  The same holds for all $C_{Y_j}$, and, considering $C_{V_2}$, all~$C_{Y'_k}$.

  Differentiating the equality giving $P$, we obtain
  \begin{align*}
  \frac{\partial P}{\partial U_{0,0}}
 & = \sum_i \left (2 C_{W_i} \frac{\partial C_{W_i}}{\partial U_{0,0}} \prod_{i'\ne i}
  C_{W_{i'}}^2  \prod_j C_{Y_j} \prod_k C_{Y'_k} \right )
  +\sum_j \left ( \frac{\partial C_{Y_j}}{\partial U_{0,0}}  \prod_i C_{W_{i}}^2  \prod_{j'\ne j} C_{Y_{j'}}
  \prod_k C_{Y'_k} \right)\\
  &+ \sum_k \left ( \frac{\partial C_{Y'_k}}{\partial U_{0,0}}  \prod_i C_{W_{i}}^2  \prod_{j} C_{Y_{j}}
  \prod_{k' \ne k} C_{Y'_{k'}} \right).
  \end{align*}
  The only possible common factors to $P$ and its partial derivative
  are the $C_{W_i}$'s (or their squares), $C_{Y_j}$'s and
  $C_{Y'_k}$'s:
  \begin{itemize}
  \item all $C_{W_i}$'s divide $\frac{\partial P}{\partial U_{0,0}}$,
    but no $C_{W_i}^2$ divides it, since otherwise $C_{W_i}$ would
    divide $2 {\partial C_{W_i}}/{\partial U_{0,0}}$. This is
    impossible, since we assume that $2$ is nonzero in $\K$, and we
    saw that $C_{W_i}$ does not divide its partial derivative.
  \item no $C_{Y_j}$ divides $\frac{\partial P}{\partial U_{0,0}}$,
    since it would imply that it divides $\frac{\partial C_{Y_k}}{\partial U_{0,0}}$,
    which we pointed out is not the case.
  \item similarly, no $C_{Y'_k}$ divides $\frac{\partial P}{\partial U_{0,0}}$.
  \end{itemize}
  We deduce that $Q = \gcd(P, \partial P/\partial U_{0,0})$ is equal
  to $c \prod_i C_{W_i}$, for some nonzero constant $c$. Because all
  $C_{W_i}$ have leading coefficient $1$, this is also the case for
  their product, so the leading coefficient of the product above is
  equal to $c$. Since $Q$ is a monic GCD, we get that $c=1$, and thus
  $Q=\prod_i C_{W_i}$. Finally, the output $P/Q$ is $\prod_i C_{W_i}
  \prod_j C_{Y_j} \prod_k C_{Y'_j}$, which we saw is the normalized
  Chow form of~$V$.
\end{proof}

\begin{lemma}\label{lemma:unionH}
  Suppose that $V_1$ and $V_2$ are algebraic sets in $\C^n$, defined over $\Q$,
  both either empty or $r$-equidimensional and in normal position,
  with respective degrees at most $D_1,D_2$ and heights at most
  $H_1,H_2$. There exists a nonzero integer $\Delta^{{\rm
      union}}_{V_1,V_2}$ with
  \[ h(\Delta^{{\rm union}}_{V_1,V_2}) \in O(n^6 (D_1+D_2)(H_1+H_2 + D_1+D_2)) \]
  such that if a prime $p$ does not divide $\Delta^{{\rm union}}_{V_1,V_2}$, then:
  \begin{itemize}
  \item $C_{V_1},C_{V_2}$ and $C=\textsc{Union\_Chow}(C_{V_1},C_{V_2})$ have coefficients in
    $\Z_{(p)}$, 
  \item $\textsc{Union\_Chow}(C_{V_1} \bmod p,C_{V_2} \bmod p)=C \bmod p$,
  \item $\gcd(C_{V_1} \bmod p,\partial C_{V_1}/\partial
    U_{0,0} \bmod p)=\gcd(C_{V_2} \bmod p,\partial C_{V_2}/\partial U_{0,0} \bmod p)=1$.
  \end{itemize}
\end{lemma}
\begin{proof}
  We first give conditions that guarantee that the first two items
  hold. The input polynomials $C_{V_1},C_{V_2}$ have respective
  degrees at most $(r+1)D_1$ and $(r+1)D_2$, and respective heights at
  most $H_1+(r+1)\ln(n+2) D_1$ and $H_2+(r+1)\ln(n+2) D_2$
  (Lemma~\ref{lemma:hchow}).

  The product $P=C_{V_1}C_{V_2}$ has degree at most $(r+1)(D_1+D_2)$
  and height at most $H=H_1+H_2 + (r+1)\ln(n+2) (D_1+D_2) +
  \ln((r+1)(n+1)+1) (D_1+D_2)$ (Lemma~1.2.1.b in~\cite{KrPaSo01}). Its
  partial derivative $\partial P/\partial U_{0,0}$ has the same degree
  bound and height at most $H + \ln(D_1+D_2)$, since $\deg(P, U_{0,0})
  \le D_1 + D_2$. In particular, both polynomials have height in
  $O(n(H_1 + H_2 +D_1 +D_2) \ln(n))$.

  
  Since these are polynomials in $O(n^2)$ variables, by
  Lemma~\ref{lemma:gcd2}, there is a nonzero integer $\Delta_0$ of
  height $O(n^6 (D_1+D_2)(H_1+H_2 + D_1+D_2))$ such that if $p$ does
  not divide $\Delta_0$, $P$ and its partial derivative have
  coefficients in $\Z_{(p)}$ and the GCD computation at step 2
  commutes with reduction modulo $p$. Note that since both $C_{V_1}$
  and $C_{V_2}$ have leading coefficient $1$, their product $P$ having
  coefficients in $\Z_{(p)}$ implies that this is also the case for
  $C_{V_1}$ and $C_{V_2}$ themselves.

  Then, Lemma~\ref{lemma:divmodp} guarantees that the division done at
  step 3 commutes with reduction modulo $p$ as well (note that $Q$
  being a monic GCD, it cannot vanish modulo $p$). Altogether, if $p$
  does not divide $\Delta_0$, the first two items hold.
  
  Finally, we give conditions to guarantee the third item.  We know
  from Lemma~\ref{lemma:gcdeq1} that $\gcd(C_{V_1},\partial
  C_{V_1}/\partial U_{0,0})=1$. Apply Lemma~\ref{lemma:gcd2} again; we
  obtain the existence of a nonzero integer $\Delta_1$, with the same
  asymptotic height bound as $\Delta_0$, such that if $p$ does not
  divide $\Delta_1$, $C_{V_1}$ has coefficients in $\Z_{(p)}$ (which
  we saw also derives from $p$ not dividing $\Delta_0$) and
  $\gcd(C_{V_1} \bmod p,\partial C_{V_1}/\partial U_{0,0} \bmod p)=1$.
  Repeat this for $C_{V_2}$, defining a nonzero integer $\Delta_2$,
  and set $\Delta^{{\rm union}}_{V_1,V_2}=\Delta_0 \Delta_1 \Delta_2$.
\end{proof}


\subsection{Scalar extension and generic projection} \label{sec:proj}

Consider an $r$-equidimensional algebraic set $V$, defined over a
field $\K$, and new variables $\ell_{i,1},\dots,\ell_{i,n}$,
$i=0,\dots,r$.  Let $\L$ be an algebraic closure of
$\K(\ell_{0,1},\dots,\ell_{r,n})$, and let $V_{\L} \subset \L^n$ be
the extension of $V$ over $\L$, that is, the zero-set of the ideal
$I(V)\cdot \L[X_1,\dots,X_n]$. The following easy observations will be
used in Subsection~\ref{ssec:sepgal}.


\begin{lemma}\label{lemma:extension}
  Let $I(V)=\cap_i P_i$ in $\Kbar[X_1,\dots,X_n]$ be the minimal
  decomposition of $I(V)$ into prime ideals. Then:
  \begin{itemize}
  \item we have $I(V_\L)=I(V)\cdot \L[X_1,\dots,X_n]$ and its minimal
    primary decomposition is $ \cap_i (P_i \cdot \L[X_1,\dots,X_n])$,
    with all ideals on the right-hand side prime.
  \item for all $i$, if $C$ is a Chow form of $V(P_i)\subset \Kbar^n$,
    then $C$ is also a Chow form of $V(P_i \cdot \L[X_1,\dots,X_n]) \subset \L^n$.
  \end{itemize}
\end{lemma}
\begin{proof}
  Set $P'_i=P_i \cdot \L[X_1,\dots,X_n]$ for all $i$.  The fact that
  all $P'_i$ remain radical
  is~\cite[Lemma~10.44.4]{stacks-project}. On another hand, it follows
  from~\cite[Lemma~10.47.2]{stacks-project} that for all $i$,
  $\L[X_1,\dots,X_n]/P'_i$ has only one minimal prime, and since it is
  reduced, it is a domain. This shows that all $P'_i$ are prime.
  Since $\L[X_1,\dots,X_n]$ is a free, hence flat,
  $\Kbar[X_1,\dots,X_n]$-algebra, extension commutes with
  intersection~\cite[Theorem~7.4.(i)]{Matsumura86}. This shows
  $I(V)\cdot \L[X_1,\dots,X_n]= \cap_i P'_i$, so this ideal is
  radical; it must then be equal to $I(V_\L)$.  Because the extension
  is also faithfully flat, $P'_i \cap \Kbar[X_1,\dots,X_n]=P_i$ for
  all $i$, so the $P'_i$ are pairwise distinct, and no $P'_i$ is
  contained in the intersection of the other ones. This shows that
  their intersection is the minimal primary decomposition of
  $I(V_\L)$.

  For the second item: a Chow form $C$ of $V(P_i)$ generates the
  elimination ideal $(P_i + (U_{j,0} + U_{j,1}X_1 + \cdots +
  U_{j,n}X_n)_{0 \le j \le r}) \cap \Kbar[\bm U_0,\dots,\bm U_r]$ (to
  verify this, note that this ideal is prime, since $P_i$ is, and each
  new equation has new leading term $U_{j,0}$), and this remains the
  case over~$\L$.
\end{proof}


The indeterminates $(\ell_{i,j})$ are used as coefficients in linear
forms $L_i= \ell_{i,1} X_1 + \cdots + \ell_{i,n} X_n$,
$i=0,\dots,r$. This allows us to define the generic projection
\[\pi: \bm x = (x_1,\dots,x_n) \mapsto (L_0(\bm x),\dots,L_r(\bm x)).\]

The image $\pi(V_{\L})$ is a Zariski-closed set in $\L^{r+1}$, which
is defined over $\K(\ell_{0,1},\dots,\ell_{r,n})$.  Our goal is to
compute a polynomial in
$\K[\ell_{0,1},\dots,\ell_{r,n}][X_1,\dots,X_n]$ whose zero-set in
$\L^n$ is the cylinder $\pi^{-1}(\pi(V_{\L}))$.

\medskip\noindent{\sc Generic\_Projection\_Chow}\vspace{-1mm}
\begin{algorithmic}[1]
  \Require polynomial $C$ in $\K[\bm U_0,\dots,\bm U_r]$, with $\bm U_i= U_{i,0},\dots,U_{i,n}$  for all $i$
  \Ensure polynomial $F$ in $\K[\ell_{0,1},\dots,\ell_{r,n}][X_1,\dots,X_n]$.
  \State \Return $C(L_0,-\ell_{0,1},\dots,-\ell_{0,n},\dots,L_r,-\ell_{r,1},\dots,-\ell_{r,n})$
\end{algorithmic}

\begin{lemma}\label{lemma:genpchow}
  Suppose that $V$ is either empty or $r$-equidimensional and defined
  over a field $\K$, and let $V_{\L}$ be the zero-set of $I(V)$ over
  an algebraic closure $\L$ of
  $\K(\ell_{0,1},\dots,\ell_{r,n})$. Given a Chow form $C$ of $V$,
  {\sc Generic\_Projection\_Chow} returns a polynomial $F$ such 
  that $V(F)=\pi^{-1}(\pi(V_{\L}))$ in~$\L^{n}$.
\end{lemma}
\begin{proof}
  When $V$ is empty, there is nothing to prove, since $C$ and thus $F$
  are nonzero constants. Otherwise, the result is proved in Proposition~15
  of~\cite{BlJeSo04}, but we can repeat the short argument using the
  notation of this paper. A point $\bm x=(x_1,\dots,x_n) \in \L^n$
  cancels $F$ if and only if $(\ell_{0,1} x_1 + \cdots + \ell_{0,n}
  x_n,-\ell_{0,1},\dots,-\ell_{0,n},\dots,\ell_{r,1} x_1 + \cdots +
  \ell_{r,n} x_n,-\ell_{r,1},\dots,-\ell_{r,n})$ cancels $C$. Now, $C$
  is also a Chow form of $V_{\L}$, so if we denote by $V_{\L}^{\rm
    proj}$ the projective closure of $V_{\L}$, this is the case if and
  only if there exists $\bm y=(y_0 : \cdots : y_n)$ in $V_{\L}^{\rm
    proj}$ such that
  \begin{align*}
    y_0(\ell_{0,1} x_1 + \cdots + \ell_{0,n} x_n) &= \ell_{0,1} y_1+ \cdots + \ell_{0,n}y_n \\
&    \dots \\
    y_0(\ell_{r,1} x_1 + \cdots + \ell_{r,n} x_n) &= \ell_{r,1} y_1+ \cdots + \ell_{r,n}y_n.
  \end{align*}
  If $y_0=0$, all $y_1 \ell_{i,1} + \cdots + y_n \ell_{i,n}$ vanish,
  and then so do all $y_0 \ell_{i,0} + y_1 \ell_{i,1} + \cdots + y_n
  \ell_{i,n}$; this implies $C(\ell_{0,0},\dots,\ell_{r,n})=0$, a
  contradiction. Thus, $F(\bm x)=0$ if and only if there exists $\bm
  y=(y_1,\dots,y_n)$ in $V_{\L}$ such that $\ell_{i,1} x_1 + \cdots +
  \ell_{i,n} x_n=\ell_{i,1} y_1 + \cdots + \ell_{i,n} y_n$ for all
  $i$, that is, such that $\pi(\bm x)=\pi(\bm y)$.
\end{proof}

\begin{remark}
  Taking coefficients of $G_{\L}$ with respect to
  $\ell_{0,1},\dots,\ell_{s,n}$ gives us a family of polynomials in
  $\K[X_1,\dots,X_n]$ whose zero-set is $W$. See also a related
  construction in~\cite{BlJeSo04}, where random linear forms are used
  instead.
\end{remark}

Note that there is no need here for $V$ to be in normal position, even
though the normal position assumption will hold when we use this
procedure. For the following lemma, though, we will make the normal position
assumption.

\begin{lemma}\label{lemma:hgenp}
  Suppose that $V$ is an algebraic set in $\C^n$, defined over $\Q$,
  either empty or $r$-equidimensional and in normal position, with
  degree at most $D$ and height at most $H$. Then
  $\textsc{Generic\_Projection\_Chow}(C_V)$ is a polynomial of degree
  at most $D$ in $X_1,\dots,X_n$ and at most $(r+1)D$ in
  $\ell_{0,1},\dots,\ell_{r,n}$, and height $O(H + n D \ln(n))$.
\end{lemma}
\begin{proof}
  The degree bounds are clear. We know that $C_V$ has height at most
  $H'=H + (r+1)D \ln(n+2)$ (Lemma~\ref{lemma:hchow}), and as in the
  proof of that lemma we can write it as the quotient $C^{\rm
    primitive}_V/d$, for some integer $d$ of height at most $H'$, with
  $C^{\rm primitive}_V/d$ the primitive Chow form of $V$, that has integer coefficients of height at
  most $H'$ as well (Lemma~\ref{lemma:hCV}). Evaluating $C_V$ at
  $L_0,-\ell_{0,1},\dots,-\ell_{0,n},\dots,L_r,-\ell_{r,1},\dots,-\ell_{r,n}$
gives us the polynomial $1/d\,
  C^{\rm primitive}_V(L_0,-\ell_{0,1},\dots,-\ell_{0,n},\dots,L_r,-\ell_{r,1},\dots,-\ell_{r,n})$.
  
  We estimate the height of the numerator as a polynomial in
  $\Z[\ell_{0,1},\dots,\ell_{r,n},X_1,\dots,X_n]$.  The polynomial
  $C^{\rm primitive}_V$ has $(r+1)(n+1) \in O(n^2)$ variables, degree at most $(r+1)D$,
  height at most $H'$; the arguments have $(r+2)n \in O(n^2)$
  variables, degree at most two and height zero. It follows from
  Lemma~1.2.1.c of~\cite{KrPaSo01} that the result has height
  $O(H'+ nD\ln(n))$.
\end{proof}


\subsection{Separation by hypersurfaces} 

Given two algebraic sets $V,Y$ in $\Kbar^n$, we define the set-theoretical operation
\[V_{\rm zero},V_{\rm proper} = \textsc{Separate}(V,Y)\]
by letting $V_{\rm zero}$ be the union of the irreducible components
of $V$ contained in $Y$, and $V_{\rm proper}$ be the union of all
other irreducible components of $V$ (either one of these two sets may
be empty).

In this paragraph, we discuss the situation where
$V$ is equidimensional and in normal position and $Y=V(G)$ is a
hypersurface.  We recall a procedure due to Jeronimo {\it et
  al.}~\cite{JeKrSaSo04} that takes as input $G$ and a normalized Chow
form of $V$ and returns the normalized Chow forms of $V_{\rm zero}$
and $V_{\rm proper}$.  We start with the following subroutine, which
takes as input a polynomial $G$ in variables $X_1,\dots,X_n$ over a
ring $\A$, and a homogenization variable $X_0$; the output is the
homogenization of $G$.

\medskip\noindent{\sc Homogenize}\vspace{-1mm}
\begin{algorithmic}[1]
  \Require polynomial $G$ in $\A[X_1,\dots,X_n]$
  \Require variables $X_0$ and $(X_1,\dots,X_n)$
  \Require integer $d$
  \Ensure homogeneous polynomial in $\A[X_0,X_1,\dots,X_n]$
  \State \Return $\sum_{\bm \alpha \in \N^n, |\bm\alpha| \le d} \coeff(G, X_1^{\alpha_1} \cdots X_n^{\alpha_n}) X_0^{d-|\bm\alpha|} X_1^{\alpha_1} \cdots X_n^{\alpha_n}$
\end{algorithmic}

\medskip

This being said, here is the procedure for separation in the
hypersurface case; it returns {\nondef} when its first input, the
polynomial $C$, is zero (which will not happen when we use it in our
main procedure). As for the union operation, we make the assumption
that $C_V$ and $\partial C_V/\partial U_{0,0}$ have no common factor,
with Lemma~\ref{lemma:gcdeq1} showing that it is always satisfied in
characteristic zero.

\medskip

\medskip\noindent{\sc Separate\_Hypersurface\_Chow}\vspace{-1mm}
\begin{algorithmic}[1]
  \Require polynomial $C$ in $\K[\bm U_0,\dots,\bm U_r]$, with $\bm U_i =U_{i,0},\dots,U_{i,n}$  for all $i$ 
  \Require polynomial $G$ in $\K[X_1,\dots,X_n]$
  \Ensure polynomials $K,L$ in $\K[\bm U_0,\dots,\bm U_r]$ or \nondef
  \State $D \gets \deg(C, U_{r,r})$
  \State $P \gets (-1)^D C(\bm v_0,\dots,\bm v_r)$ with $\bm v_i = U_{i,0}-T_i,U_{i,1},\dots,U_{i,n}$ for all $i$, $T_0,\dots,T_r$ new variables
  \State $G^h \gets \textsc{Homogenize}(G,X_0,(X_1,\dots,X_n), \deg(G))$
  \LeftComment{homogenization of $G$ using a new  variable $X_0$}
  \State $ G[P] \gets G^h \left(\frac{\partial P}{\partial T_0},-\frac{\partial P}{\partial U_{0,1}},\dots,-\frac{\partial P}{\partial U_{0,n}}
  \right )$
  \State $Q \gets \gcd (P, G[P])$
  \State $R \gets P/Q$
  \State \Return $(-1)^{\deg(Q,T_0)}Q(T_0 = \cdots =T_r=0),\ (-1)^{\deg(R,T_0)}R(T_0 = \cdots =T_r=0)$
\end{algorithmic}

 \begin{lemma}\label{lemma:separate}
  Let $V$ be either empty or $r$-equidimensional and in normal
  position, and defined over a field $\K$. Given its normalized Chow
  form $C_V$ and $G$ in $\K[X_1,\dots,X_n]$, and assuming that
  $\gcd(C_V, \partial C_V/\partial U_{0,0})=1$, {\sc
    Separate\_Hypersurface\_Chow} returns the normalized Chow forms of
  respectively $V_{\rm zero},V_{\rm proper}=\textsc{separate}(V,V(G))$.
\end{lemma}
 \begin{proof} The case where $V$ is empty is clear (since then $C=P=Q=R=1$).
   Else, let us first establish that
   the polynomials $P$ and $\partial P/\partial T_0$ have constant GCD.

   Suppose this is not the case, so there is an irreducible factor $Q$
   of $P$ that divides $\partial P/\partial T_0$. By construction, the
   irreducible factors of $P$ are in one-to-one correspondence with
   those of $C_V$, and in particular $P$ has no multiple factor (since
   $C_V$ does not). This implies that $Q$ must divide $\partial
   Q/\partial T_0$, and thus that $\partial Q/\partial T_0=0$. On the
   other hand, we know that $Q=(\pm 1) B(\bm v_0,\dots,\bm v_r)$, with
   $B$ an irreducible factor of $C_V$ and $\bm v_i =
   U_{i,0}-T_i,U_{i,1},\dots,U_{i,n}$ for all $i$.  This gives
   $\partial Q/\partial T_0 = (\pm 1) \partial B/\partial U_{0,0} (\bm
   v_0,\dots,\bm v_r)$, so $\partial B/\partial U_{0,0}=0$. But then
   $B$ is a common factor to $C_V$ and $\partial C_V/\partial
   U_{0,0}$, in contradiction with our assumption $\gcd(C_V, \partial
   C_V/\partial U_{0,0})=1$.

   The fact that $\gcd(P,\partial P/\partial T_0)=1$ is equivalent to
   saying that $\partial P/\partial T_0$ does not vanish identically
   on any irreducible component of $V(P)$. Knowing this, the proof
   of~\cite[Lemma~3.9]{JeKrSaSo04} establishes that $Q=\gcd (P, G[P])$
   is a characteristic polynomial of $V_{\rm zero}$ (that reference
   assumes characteristic zero, where such issues do not arise).
   Since $V$ is in normal position, it is also the case for $V_{\rm
     zero}$ (Lemma~\ref{lemma:normal}), so the latter admits a unique
   normalized characteristic polynomial. Since both normalized
   characteristic polynomial and monic GCDs have leading coefficient
   $1$ (the former by Lemma~\ref{lemma:defcharpoly}), we see that $Q$
   is actually this normalized characteristic polynomial.

   On the other hand, $P$ itself is the normalized characteristic
   polynomial of $V$. It follows from Lemma~\ref{lemma:normalcharpoly}
   that $R$ is the normalized characteristic polynomial of $V_{\rm
     proper}$.  Evaluating all $T_i$'s at $0$ then gives us Chow forms
   of $V_{\rm zero}$ and $V_{\rm proper}$. To normalize them, it
   suffices to adjust their sign by $(-1)^{\deg(V_{\rm zero})}$ and
   $(-1)^{\deg(V_{\rm proper})}$, and by
   Lemma~\ref{lemma:defcharpoly}, these degrees are precisely the
   degrees of $Q$ and $R$ in $T_0$.
\end{proof}

 \begin{lemma}\label{lemma:HsepH}
   Suppose that $V$ is an algebraic set in $\C^n$, defined over $\Q$,
   either empty or $r$-equidimensional and in normal position, with
   degree at most $D$ and height at most $H$; suppose also that $G$
   has degree at most $d$ and height at most $h$. There exists a
   nonzero integer $\Delta^{{\rm separate\_H}}_{V,G}$ with
   \[h(\Delta^{\rm separate\_H}_{V,G})
   \in O( n^6 D ( d\ln(d)  + h +  dH + d D ) )\]
   such that if a prime $p$ does not divide $\Delta^{{\rm separate\_H}}_{V,G}$, then:
   \begin{itemize}
   \item $C_V$, $G$ and $K,L=\textsc{Separate\_Hypersurface\_Chow}(C_V, G)$ have coefficients in $\Z_{(p)}$
   \item $\textsc{Separate\_Hypersurface\_Chow}(C_V \bmod p, G \bmod p)=K \bmod p, L \bmod p$
   \item $\gcd(C_V \bmod p, \partial C_V/\partial U_{0,0} \bmod p)=1$.
   \end{itemize}
 \end{lemma}
 \begin{proof}
   Let us rewrite $C_V = C'/c$ and $G=G'/g$, with $c,g$ integers,
   $C'=C^{\rm primitive}_V,G'$ with integer coefficients, $c,C'$ of
   height at most $H'=H+(r+1)D\ln(n+2)$
   (Lemmas~\ref{lemma:hCV},~\ref{lemma:hchow}) and $g,G'$ of height at
   most $h$. Our first requirement is that $p$ does not divide $\Delta_0 =
   c\, g\, \lc(G')$. The nonvanishing of the first two factors ensures
   that $C_V$ and $G$ have coefficients in $\Z_{(p)}$, the latter one
   that the degree of $G$, or $G'$, does not drop modulo $p$ (since we
   use a graded order).
   \begin{itemize}
   \item At step 1, the polynomials $C_V$ and $C_V \bmod p$ have the
     same degree in $U_{r,r}$, since $C_V$ has leading term $U_{0,0}^D
     \cdots U_{r,r}^D$.
   \item Step 2 computes $P=P'/c$, with $P'= (-1)^D C'(\bm
     v_0,\dots,\bm v_r)$ of degree at most $(r+1)D$ and with integer
     coefficients. We know that $C'$ has height at most $H'$, so $P'$
     has height at most $H'+ r D$.
   \item At step 3, because the degree of $G$ does not drop modulo
     $p$, we have the equality $(G \bmod p)^h=G^h \bmod p$. Besides,
     $G^h$ can be written as $G^h = {G'}^h/g$, with ${G'}^h$ the
     homogenization of $G'$, which is homogeneous of degree at most $d$, with
     integer coefficients of height at most $h$.
   \item At step 4, we have $G[P]=G'[P]/g= 1/(c^{\deg(G)} g) G'[P']$.
     All partial derivatives of $P'$ have degree at most $(r+1)D$ and
     height at most $H''=H' + r D+\ln(D)$, so the
     polynomial $G'[P']$ has degree at most $(r+1)dD$, and height at
     most $h + d(H'' + \ln(n+2) + \ln((r+1)(n+2))(r+1) D)$, by
     Lemma~1.2.1.c of~\cite{KrPaSo01}.
   \item We can now invoke Lemma~\ref{lemma:gcd2} for the polynomials
     $P'$ and $G'[P']$, with $d_F \in O(nD)$, $h_F \in O(H+n \ln(n)D)$,
     $d_G \in O(ndD)$, $h_G \in O(h+ dH + dn \ln(n)D)$, and $s \in O(n^2)$.
     It shows the existence of a nonzero integer $
     \Delta_1$ with
     \[h(\Delta_1)
     \in O( n^6 D ( d\ln(d) + h + dH + d D ) )\] such that if $p$ does
     not divide $\Delta_0 \Delta_1$, step 5 reduces well modulo $p$,
     showing in particular that the result of this step has
     coefficients in $\Z_{(p)}$.
   \item The quotient at step 6 is then handled by
     Lemma~\ref{lemma:divmodp} (since the GCD $Q$ is monic, it does
     not vanish modulo $p$) and clearing the variables $T_i$ in $Q$
     and $R$ at step 7 raises no difficulty. Since $Q$ and $R$ have
     leading monomial $(T_0 U_{1,1} \cdots U_{r,r})^{\deg(Q,T_0)}$ and
     $(T_0 U_{1,1} \cdots U_{r,r})^{\deg(R,T_0)}$ with leading
     coefficient $1$ (Lemma~\ref{lemma:defcharpoly}), the degrees
     $\deg(Q,T_0)$ and $\deg(R,T_0)$ do not change when we run the
     algorithm modulo $p$.
   \end{itemize}
   It follows that if $p$ does not divide $\Delta_0 \Delta_1$, the
   first two properties in the lemma hold. To establish the last one,
   as in the proof of Lemma~\ref{lemma:unionH}, $p$ should avoid
   dividing a nonzero integer $\Delta_2$ of height $O(n^6 D(D+H))$. We take
   $ \Delta^{\rm separate\_H}_{V,G}=\Delta_0 \Delta_1 \Delta_2$, and the lemma is proved.
 \end{proof}
 

\subsection{Separation in general} \label{ssec:sepgal}

Now, we discuss the situation where both inputs $V,W$ are given by
means of their Chow form, so in particular both $V$ and $W$ are
equidimensional, but we drop the assumption that $W$ be a
hypersurface. The procedure is given below; since it uses
$\textsc{Separate\_Hypersurface\_Chow}$, it may return \nondef, but we
prove below that for legitimate inputs (Chow forms of algebraic sets,
in normal position for the first of them), the output is correct.

\medskip\noindent{\sc Separate\_Chow}\vspace{-1mm}
\begin{algorithmic}[1]
  \Require polynomial $C$ in $\K[\bm U_0,\dots,\bm U_r]$, with $\bm U_i =U_{i,0},\dots,U_{i,n}$  for all $i$ 
  \Require polynomial $D$ in $\K[\bm U_0,\dots,\bm U_s]$, with $\bm U_i =U_{i,0},\dots,U_{i,n}$  for all $i$ 
  \Ensure polynomials $Q,R$ in $\K[\bm U_0,\dots,\bm U_r]$ or \nondef
  \State $G_{\L} \gets \textsc{Generic\_Projection\_Chow}(D)$
  \State \Return $\textsc{Separate\_Hypersurface\_Chow}(C, G_{\L})$
\end{algorithmic}

\begin{lemma}\label{lemma:separategeneral}
  Let $V$ be either empty or $r$-equidimensional and in normal
  position, and let $W$ be either empty or $s$-equidimensional, both defined
  over a field $\K$. Given
  the normalized Chow form $C_V$ of $V$ and a Chow form $D$ of $W$,
  and assuming that $\gcd(C_V, \partial C_V/\partial U_{0,0})=1$, {\sc
    Separate\_Chow} returns the normalized Chow forms of respectively
  $V_{\rm zero},V_{\rm proper}=\textsc{Separate}(V,W)$.
\end{lemma}
\begin{proof}
  Let $\pi$ be as in Section~\ref{sec:proj}, let $W_{\L}$ be zero-set
  of $I(W)$ over an algebraic closure $\L$ of
  $\K(\ell_{0,1},\dots,\ell_{s,n})$, and let
  $G_{\L}=\textsc{Generic\_Projection\_Chow}(D)$. By
  Lemma~\ref{lemma:genpchow}, $V(G_{\L})=\pi^{-1}(\pi(W_{\L}))$, where both left
  and right-hand sides are in $\L^{n}$.

  For $\bm x$ in $\Kbar^n$, we claim that $\bm x$ is in $W$ if and
  only if it is in $V(G_{\L})$ (of course, this does not mean that $W=V(G_\L)$, since
  this latter set lies in $\L^n$). If $\bm x$ is in $W$, it is also in
  $W_\L$, and thus in $V(G_{\L})$. The converse follows from Proposition~3
  of~\cite{BlJeSo04}: given $\bm x$ not in $W$, that proposition
  establishes the existence of a nonempty Zariski-open set $O_{\bm x,W}$
  in $\Kbar^{(n+1)(r+1)}$ such that for
  $(\lambda_{0,0},\dots,\lambda_{r,n})$ in $O_{\bm x,W}$, if we set
  \[\pi_{\lambda}: \bm z \mapsto  (\lambda_{0,0} + \lambda_{0,1} z_n + \cdots + \lambda_{0,n} z_n,
  \dots,\lambda_{r,0} + \lambda_{r,1} z_n + \cdots + \lambda_{r,n}
  z_n),\] then $\pi_\lambda(\bm x)$ is not in $\pi_\lambda(W)$. If
  $\pi(\bm x)$ is in $\pi(W_\L)$, there exist $\bm y =(y_1,\dots,y_n)$ in
  $W_\L$ such that
  \[\ell_{i,1} x_n + \cdots + \ell_{i,n} x_n = \ell_{i,1} y_1 + \cdots + \ell_{i,n} y_n,\quad i=0,\dots,r.\]
  From this, it would follow that for generic $(\lambda_{0,0},\dots,\lambda_{r,n})$
  in $\Kbar^{(n+1)(r+1)}$, there exists $\bm z=(z_1,\dots,z_n)$ in $W$
  with
  \[\lambda_{i,1} x_n + \cdots + \lambda_{i,n} x_n = \lambda_{i,1} z_1 + \cdots + \lambda_{i,n} z_n,\quad i=0,\dots,r,\]
  and thus $\pi_\lambda(\bm x)=\pi_\lambda(\bm z)$, a contradiction. So 
  $\pi(\bm x)$ is not in $\pi(W_\L)$, and by definition of $G_\L$ this means
  that $G_\L(\bm x)\ne 0$.
  
  Now, consider an irreducible component $V'_\L$ of $V_\L$. We saw in
  Lemma~\ref{lemma:extension} that there exists an irreducible
  component $V'$ of $V$ such that $I(V'_\L)=I(V')\cdot
  \L[X_1,\dots,X_n]$. If $V'$ is contained in $W$, it follows that
  $V'_\L$ is contained in $W_\L$, and thus in
  $\pi^{-1}(\pi(W_{\L}))=V(G_{\L})$. If $V'$ is not contained in $W$,
  it contains a point $\bm x \in \Kbar^n$ that does not belong to
  $W$. Then, by the claim in the previous paragraph, $\bm x$ is not in
  $V(G_\L)$, and thus $V'$, as well as $V'_\L$, are not contained in
  $V(G_{\L})$.

  In other words, if we set $Y,Y'=\textsc{Separate}(V_\L,V(G_\L))$,
  $Y$ is the union of all these $V'_\L$ for which $V' \subset W$, and
  $Y'$ is the union of all other $V'_\L$. From the second
  point of Lemma~\ref{lemma:extension}, we deduce that
  their normalized Chow forms are those of 
  respectively $V_{\rm zero},V_{\rm proper}=\textsc{Separate}(V,W)$.

  Since we assume that $V$ is either empty or in normal position with
  $\gcd(C_V, \partial C_V/\partial U_{0,0})=1$, this is also the case
  for $V_\L$, since $C_V$ remains a Chow form for it (this also follows from the
  second point of Lemma~\ref{lemma:extension}). We can then apply
  Lemma~\ref{lemma:separate} to $C_V$ and~$G_\L$, and conclude the proof.
\end{proof}

\begin{lemma}\label{lemma:Hsep}
  Suppose that $V$ and $W$ are algebraic sets in $\C^n$, defined over
  $\Q$, both either empty or $r$-equidimensional
  (resp.\ $s$-equidimensional) and in normal position, with respective
  degrees at most $D_V,D_W$ and heights at most $H_V,H_W$. There
  exists a nonzero integer $\Delta^{{\rm separate}}_{V,W}$ with
  \[h(\Delta^{\rm separate}_{V,W})
  \in O( n^6 D_V ( D_W \ln(D_W) + H_W + D_W H_V +D_V D_W ))\]
  such that if a prime $p$ does not divide $\Delta^{{\rm separate}}_{V,W}$, then:
  \begin{itemize}
  \item $C_V$, $C_W$ and $K,L=\textsc{Separate\_Chow}(C_V, C_W)$ have coefficients in $\Z_{(p)}$
  \item $\textsc{Separate\_Chow}(C_V \bmod p, C_W \bmod p)=K \bmod p, L \bmod p$
   \item $\gcd(C_V \bmod p, \partial C_V/\partial U_{0,0} \bmod p)=1$.
  \end{itemize}
\end{lemma}
\begin{proof}
  The proof is similar to that of Lemma~\ref{lemma:HsepH}, taking
  $G=G_{\L}$. We introduce the same polynomials $C',P',G'$ and
  ${G'}^h$, with $P'$ still having degree $O(nD_V)$ and height $O(H_V
  + n D_V \ln(n))$; now, by Lemma~\ref{lemma:hgenp}, ${G'}^h$ has
  degree $O(D_W)$ in $X_1,\dots,X_n$, $O(n D_W)$ in
  $\ell_{0,1},\dots,\ell_{s,n}$, and height $O(H_W+n D_W \ln(n))$.
  Again, the first property to guarantee is that the denominators in
  $P'$ and $G'$ do not vanish modulo $p$, and that reducing $G$ modulo
  $p$ does not change its degree. This is done as in
  Lemma~\ref{lemma:HsepH}: it happens as long as $p$ does not divide a
  certain nonzero integer $\Delta_0$ of height $O(H_V + H_W + n
  (D_V+D_W) \ln(n))$.

  We are left with giving conditions for the computation of the monic
  GCD, say $Q'$, of $P'$ and $G'[P']$ to reduce well modulo $p$. This
  is a GCD over the ground field $\Q(\ell_{0,1},\dots,\ell_{s,n})$, in
  variables $T_i,U_{i,0},\dots,U_{i,n}$ for $i=0,\dots,r$. However,
  since $P'$ is in $\Q[T_0,\dots,U_{r,n}]$, $Q'$ lies in
  $\Q[T_0,\dots,U_{r,n}]$, and thus is also the monic GCD of $P'$ and
  $G'[P']$ in $\Q[\ell_{0,1},\dots,\ell_{s,n},T_0,\dots,U_{r,n}]$ as
  well.  Similarly, for any prime $p$ (that does not divide
  $\Delta_0$), the GCD of $P' \bmod p$ and $G'[P'] \bmod p$ can be
  computed in $\F_p[\ell_{0,1},\dots,\ell_{s,n},T_0,\dots,U_{r,n}]$.
  
  We can then again apply Lemma~\ref{lemma:gcd2}, this time to the polynomials
  $P'$ and $G'[P']$ in
  $\Q[\ell_{0,1},\dots,\ell_{s,n},T_0,\dots,U_{r,n}]$. These are
  polynomials in $O(n^2)$ variables, with the former of degree
  $O(nD_V)$ and height $O(H_V + n D_V \ln(n))$, and the latter of
  degree $O(n D_V D_W)$ and height $O(H_W + n D_W H_V + n^2 D_V D_W \ln(n))$. The lemma then implies the existence of
  an integer $\Delta_1$ of height
  \[
  h(\Delta_1) \in O( n^6 D_V ( D_W \ln(D_W) + H_W + D_W H_V +D_V D_W ))
  \]
  such that if $p$ does not divide $\Delta_0 \Delta_1$, the GCD
  computation reduces well modulo $p$ (so the result of this step has
  coefficients in $\Z_{(p)}$). The final division is again
  handled by Lemma~\ref{lemma:divmodp}.

  As in Lemma~\ref{lemma:HsepH}, we also need to guarantee that
  $\gcd(C_V\bmod p, \partial C_V/\partial U_{0,0}\bmod p)=1$; this is
  the case when $p$ does not divide a nonzero integer $\Delta_2$ of
  height $O(n^6 D_V (D_V + H_V))$. We take $\Delta^{\rm
    separate}_{V,W}=\Delta_0 \Delta_1\Delta_2$; its height is bounded
  above by the same asymptotic expression as that of $\Delta_1$.
\end{proof}


\subsection{Computing a degree $d$ Chow form} \label{sec:degd}

Take $V$ $r$-equidimensional, in normal position and of degree $D$,
and assume that we know its normalized Chow form $C_V$. For $F$ of
degree $d \ge 1$, we recall another result from~\cite{JeKrSaSo04}, for
the computation of $C_{d,V}[F]$ (see Section~\ref{sec:defchow} for the notation),
where $C_{d,V}$ is normalized degree $d$ Chow form of $V$.

We need a straightforward subroutine \textsc{Pseudo-companion}, which
takes as input a polynomial $P(T)$ of degree $D$ and returns the
companion matrix of $P/a$, multiplied by $a$ (no division or
multiplication is needed, all entries are coefficients of $P$).

In the algorithm, we use a division (at Step~\ref{step:exdiv}) in a
multivariate polynomial ring. In the previous algorithms we saw, all
divisions were guaranteed to be exact, since they were all of the form
$P/\gcd(P,\dots)$. In this instance, this may not be the case for all
inputs, so to cover all possible situations, we return {\nondef} when
this division is not exact. As usual, in the lemma below, we 
mention that the division is exact in all cases of interest to us.

\medskip\noindent{\sc Degree\_$d$\_Chow\_form}\vspace{-1mm}
\begin{algorithmic}[1]
  \Require polynomial $C$ in $\K[\bm U_0,\dots,\bm U_r]$, with $\bm U_i =U_{i,0},\dots,U_{i,n}$ for all $i$
  \Require polynomial $F$ in $\K[X_1,\dots,X_n]$
  \Ensure polynomial $Q$ in $\K[\bm U_0,\dots,\bm U_{r-1}]$ or \nondef
  \State $d \gets \deg(F)$
  \State $D \gets \deg(C,U_{0,0})$
  \State $P \gets (-1)^D C(\bm U_0,\dots,\bm U_{r-1}, \bm U'_r) \in \K[\bm U_0,\dots,\bm U_r][T]$, with $\bm U'_r=U_{r,0}-T,U_{r,1},\dots,U_{r,n}$
  \State $a \gets \coeff(P, T^D) \in \K[\bm U_0,\dots,\bm U_r]$
  \State $M_P \gets \textsc{Pseudo-companion}(P)$
  \State $w_0 \gets \frac{\partial P}{\partial T}$
  \State \textbf{for} $i=1,\dots,n$ \textbf{do} $w_i\gets -\frac{\partial P}{\partial U_{r,i}}$
  \State \textbf{for} $i=0,\dots,n$ \textbf{do} $w_i^h \gets \textsc{Homogenize}(w_i, T',(T), D)$
  \LeftComment{homogenization of $w_i$  in degree $D$
  with respect to $T$, using a new variable $T'$}
  \State \textbf{for} $i=0,\dots,n$  \textbf{do} $M_i \gets w^h_i(a {\rm Id}, M_P)$
  \State\label{step:beforeexdiv} $F_{\rm gen} \gets \sum_{|\bm\alpha| \le d} V_\alpha X_0^{d-|\alpha|} X_1^{\alpha_1} \cdots X_n^{\alpha_n}$,
  $\bm V=(V_{\bm\alpha})_{|\bm \alpha|\le d}$ new indeterminates
  \LeftComment{generic homogeneous polynomial of degree $d$}
  \State\label{step:exdiv}  $G \gets {a^d \det(F_{\rm gen}(M_0,\dots,M_n))} / {\det(M_0)^d}.$
  \State \Return $G(\bm U_0,\dots,\bm U_{r-1},\acoeff(F))$
\end{algorithmic}

\medskip

\begin{lemma}\label{lemma:degDCF}
  Let $V$ be either empty or $r$-equidimensional and in normal
  position, and defined over a field $\K$. Given its normalized Chow
  form $C_V$ and $F$ in $\K[X_1,\dots,X_n]$, and assuming that
  $\gcd(C_V, \partial C_V/\partial U_{0,0})=1$, {\sc
    Degree\_$d$\_Chow\_form} returns $C_{\deg(F),V}[F]$, with the
  division at Step~\ref{step:exdiv} being exact.
\end{lemma}
\begin{proof}
  The proof is in~\cite[Section~3.2]{JeKrSaSo04}; here, we simply highlight
  where the assumption that $\gcd(C_V, \partial C_V/\partial U_{0,0})=1$ is used.

  For the division at Step~\ref{step:exdiv} to be well-defined, we
  have to guarantee that $\det(M_0)$ does not vanish (then, the
  division being exact is established in the reference above). This
  determinant is, up to a power of $a=\coeff(P, T^D)$, the resultant
  of $P$ and $\frac{\partial P}{\partial T}$ with respect to $T$.  It
  is then enough to verify that $\gcd(P,\partial P/\partial T)=1$ in
  $\K(\bm U_0,\dots,\bm U_r)[T]$.

  Because $C_V$ is symmetric in $\bm U_0,\dots,\bm U_r$, we have
  $\gcd(C_V, \partial C_V/\partial U_{r,0})=1$. We can then repeat the
  argument used in the proof of Lemma~\ref{lemma:separate}, and prove
  that $\gcd(P,\partial P/\partial T)=1$ in $\K[\bm U_0,\dots,\bm
    U_r,T]$. This implies that these two polynomials also have
  constant GCD in $\K(\bm U_0,\dots,\bm U_r)[T]$.
\end{proof}

\begin{lemma}\label{lemma:deltachow}
  Suppose that $V$ is an algebraic set in $\C^n$, defined over $\Q$,
  either empty or $r$-equidimensional and in normal position, with
  degree at most $D_V$ and height at most $H_V$, and that $F$ has integer
  coefficients, with degree at most $d_F$ and height at most $h_F$. Then,
  there exists a nonzero integer $\Delta^{{\rm Chow}}_{V,F}$ with
  \[h(\Delta^{\rm Chow}_{V,F})  \in O(  n^6 D_V(D_V + H_V)  + h_F)\]
  such that if a prime $p$ does not divide $\Delta^{{\rm Chow}}_{V,F}$, then:
  \begin{itemize}
  \item $C_V$ and $Q=\textsc{Degree\_$d$\_Chow\_form}(C_V, F)$ have coefficients in $\Z_{(p)}$,
  \item $\textsc{Degree\_$d$\_Chow\_form}(C_V \bmod p, F \bmod p)=Q \bmod p$,
  \item $\gcd(C_V \bmod p, \partial C_V/\partial U_{0,0} \bmod p)=1$.
  \end{itemize}
\end{lemma}
\begin{proof}
  As before, we rewrite $C_V$ as $C_V=C^{\rm primitive}_V/c$, with $c$
  a minimal denominator of $C_V$, and we let $\Delta_0$ be the product
  of $c$ by the coefficient of any monomial of maximal degree in
  $F$. Since $C_V$ has height at most $H'=H_V + (r+1)\ln(n+2)D_V \in
  O(H_V+nD_V\ln(n))$ and $F$ at most $h_F$, $\Delta_0$ has height $O(H_V +n D_V
  \ln(n) + h_F)$.

  As in the proof of Lemma~\ref{lemma:unionH}, there exists
  $\Delta_1$ nonzero, with $h(\Delta_1) \in O(n^6 D_V(D_V + H_V))$,
  such that if $p$ does not divide $\Delta_0\Delta_1$, then in addition,
  $\gcd(C_V \bmod p, \partial C_V/\partial U_{0,0} \bmod p)=1$.
  Suppose then that a prime $p$ does not divide $\Delta_0\Delta_1$, so
  in particular $C_V$ has coefficients in $\Z_{(p)}$. The integers $d$
  and $D$ computed at the first two steps remain the same modulo $p$
  (for the former, this is by construction of $\Delta_0$, for the
  latter this is because $C_V$ has leading coefficient $1$). This
  implies that all quantities computed up to
  Step~\ref{step:beforeexdiv} have coefficients in $\Z_{(p)}$, and
  that their computation commute with reduction modulo $p$. It follows
  that $A=a^d \det(F_{\rm gen}(M_0,\dots,M_n))$ and $B=\det(M_0)^d$
  have coefficients in $\Z_{(p)}$, and we know from the previous lemma
  that $B$ is nonzero and divides $A$.

  To be able to apply Lemma~\ref{lemma:divmodp}, we have to ensure
  that $\det(M_0)$ does not vanish modulo $p$, but we saw in the proof
  of the previous lemma that this is the case, since $\gcd(C_V \bmod
  p, \partial C_V/\partial U_{0,0} \bmod p)=1$.  Then,
  Lemma~\ref{lemma:divmodp} ensures that $G$ computed at
  Step~\ref{step:exdiv} has coefficients in $\Z_{(p)}$, and that
  division commutes with reduction modulo $p$. Finally, specialization
  at the coefficients of $F$ raises no difficulty.

  Altogether, we set $\Delta^{\rm Chow}_{V,f}=\Delta_0 \Delta_1$,
  which has height $O(n^6 D_V(D_V + H_V)  + h_F)$, and the lemma is proved.
\end{proof}


\subsection{Proper intersection} \label{ssec:propint}

Finally, take $V$ $r$-equidimensional in normal position and $F$ in
$\K[X_1,\dots,X_n]$ and assume now that $W=V \cap V(F)$ has dimension
$r-1$ (or is empty). We recall here another procedure
from~\cite{JeKrSaSo04}, itself using a result due to
Philippon~\cite{Philippon86}, that gives us a Chow form of $W$. 
We start with two definitions.

\begin{definition}
Suppose that $Q$ is a nonzero polynomial
in $\K[\bm U_0,\dots,\bm U_{r-1}]$, and consider its absolute
factorization, that is, its factorization into monic irreducibles in
$\Kbar[\bm U_0,\dots,\bm U_{r-1}]$. Separate those irreducible factors
that depend on $U_{0,0}$ from those that don't, to get an expression
\begin{align}\label{eq:factoQ}
 Q = \lc(Q) \prod_{i \in I_1} C_i^{m_i} \prod_{i \in I_2} C_i^{m_i},
\end{align}
with positive multiplicities $m_i$, and with $i \in I_1$ if $C_i$ does
not depend on $U_{0,0}$ and $i \in I_2$ if it does. The {\em weak
  squarefree part} $S(Q)$ of $Q$ is the product $C=\prod_{i \in I_2}
C_i$.
\end{definition}

This is a monic polynomial, {\it a priori} with coefficients in
$\Kbar$. If $\K$ is a perfect field, however, factoring over $\K$ is
enough: starting from an irreducible factorization as
in~\eqref{eq:factoQ}, but over $\K$, also yields $S(Q)$.


\begin{definition}
  Suppose that $Q$ is a nonzero polynomial in $\K[\bm U_0,\dots,\bm
    U_{r-1}]$. We say that $Q$ is {\em tame} if $C=S(Q)$ satisfies
  $\gcd(C, \partial C/\partial U_{0,0})=1$, and if the multiplicities
  in~\eqref{eq:factoQ} are all units in~$\K$.
\end{definition}
If $\K$ is perfect, the multiplicities obtained from factoring $Q$
over $\Kbar$ are the same as those obtained from factorization over
$\K$. If $\K$ has characteristic zero, any nonzero $Q$ is tame.

\begin{lemma}\label{lemma:tame}
  Suppose that $Q$ is tame. Then $S(Q) = Q/ ( \lc(Q)
  \gcd(Q, \partial Q/\partial U_{0,0}))$.
\end{lemma}
\begin{proof}
  Write $C=S(Q)$.
  Starting from $C = \prod_{i \in I_2} C_i$, our first assumption
  implies that $\gcd(C_i,\partial C_i/\partial U_{0,0})=1$ for all $i$
  in $I_2$. Differentiating the equation defining $Q$
  gives
  \[ \frac{\partial Q}{\partial U_{0,0}} = \lc(Q) \left (\prod_{i \in I_1} C_i^{m_i} \right ) \sum_{i \in I_2} m_i
  \frac{\partial C_i}{\partial U_{0,0}} C_i^{m_i-1}\prod_{i' \in I_2-\{i\}} C_{i'}^{m_{i'}}.\]
  Since all $m_i$'s are nonzero in $\K$, it
  follows that $\gcd(Q,\partial Q/\partial U_{0,0})= \prod_{i \in
    I_1} C_i^{m_i} \prod_{i \in I_2} C_i^{m_i-1}$ (both sides are monic), and finally that
  \[\frac{Q}{\gcd(Q,\partial Q/\partial U_{0,0})} = \lc(Q) C\qedhere\]
\end{proof}

The point in this definition comes
from~\cite[Proposition~2.4]{Philippon86}
and~\cite[Lemma~3.7]{JeKrSaSo04}: if $V$ is $r$-equidimensional in
normal position and $F$ in $\K[X_1,\dots,X_n]$, and if $W=V \cap V(F)$
has dimension $r-1$ (or is empty), then $C=S(C_{\deg(F),V}[F])$ is a
Chow form of $W$ (and in particular $C=C_W$ if $W$ is in normal
position, since $C$ is monic).

This leads us to the following procedure. Given arbitrary input, it
may return {\nondef}, if \textsc{Degree\_$d$\_Chow\_form} does or if
$Q$ vanishes (contradicting the assumption above), since then the
monic GCD at Step 3 is undefined. The following lemma summarizes
conditions that guarantee a valid output.

\medskip\noindent{\sc Intersection\_Chow}\vspace{-1mm}
\begin{algorithmic}[1]
  \Require polynomial $C$ in $\K[\bm U_0,\dots,\bm U_r]$, with $\bm U_i =U_{i,0},\dots,U_{i,n}$ for all $i$
  \Require polynomial $F$ in $\K[X_1,\dots,X_n]$ 
  \Ensure  polynomial $D$ in $\K[\bm U_0,\dots,\bm U_{r-1}]$ or \nondef
  \State $Q \gets\textsc{Degree\_$d$\_Chow\_form}(C, F)$
  \State $R \gets \gcd(Q, \partial Q/\partial U_{0,0})$
  \State $S \gets Q/R$
  \State \Return $S/\lc(Q)$
\end{algorithmic}

\begin{lemma}\label{lemma:interchow}
  Let $V$ be an algebraic set defined over a field $\K$, and let $F$ be in
  $\K[X_1,\dots,X_n]$. Suppose that the following hold:
  \begin{itemize}
  \item $V$ is either empty or $r$-equidimensional and in normal position
  \item $W=V \cap V(F)$ is either empty or $(r-1)$-equidimensional
  \item $\gcd(C_V,\partial C_V/\partial U_{0,0})=1$
  \item $\textsc{Degree\_$d$\_Chow\_form}(C_V, F)$ is tame.
  \end{itemize}
  Then, given the normalized Chow form $C_V$ of $V$ and $F$, {\sc
    Intersection\_Chow} returns a Chow form of $W$.
   If $W$ is in normal position, the output is its normalized Chow
   form $C_W$.
\end{lemma}
\begin{proof}
  Using Lemma~\ref{lemma:degDCF}, the first and third assumptions imply that
  at Step 1, $Q=C_{\deg(F),V}[F]$. The assumption on $W$ gives that
  $S(C_{\deg(F),V}[F])$ is a Chow form of $W$, and the last assumption
  allows us to apply the previous lemma to compute it.
\end{proof}

\begin{lemma}\label{lemma:interchowH}
  Let $V$ be an algebraic set defined over $\Q$, and let $F$ be in
  $\Z[X_1,\dots,X_n]$. Suppose that the following hold:
  \begin{itemize}
  \item $V$ is either empty or $r$-equidimensional and in normal position,
  \item $W=V \cap V(F)$ is either empty or $(r-1)$-equidimensional and in normal position.
  \end{itemize}
  Assume that $V$ has degree at most $D_V$ and height at most $H_V$, and
  $F$ has degree at most $d_F$ and height at most $h_F$. Then, there
  exists a nonzero integer $\Delta^{{\rm interesection}}_{V,F}$ with
  \[h(\Delta^{\rm intersection}_{V,F})
  \in O( n^6 d_F^2 D_V H_V + n^6 d_F h_F D_V^2  + n^7 d_F^2  D_V^2\ln(n) )\]
  such that if a prime $p$ does not divide $\Delta^{{\rm intersection}}_{V,F}$, then:
  \begin{itemize}
  \item $C_V$ and $\textsc{Intersection\_Chow}(C_V, F)$ have coefficients in $\Z_{(p)}$
  \item $\textsc{Intersection\_Chow}(C_V \bmod p, F \bmod p)=\textsc{Intersection\_Chow}(C_V, F) \bmod p$
  \item $\gcd(C_V \bmod p,\partial C_V/\partial U_{0,0} \bmod p)=1$
  \item $\textsc{Degree\_$d$\_Chow\_form}(C_V \bmod p, F \bmod p)$ is tame.
  \end{itemize}
\end{lemma}
\begin{proof}
  First, note that since $V$ is defined over $\Q$, the third and
  fourth assumptions in the previous lemma are automatically satisfied
  (Lemma~\ref{lemma:gcdeq1} for the former, and remark before
  Lemma~\ref{lemma:tame} for the latter, since $C_{\deg(F),V}[F]$ is
  nonzero), so that lemma proves that
  $C_W=\textsc{Intersection\_Chow}(C_V, F)$.
  
  Our first constraint on $p$ is that it does not divide the integer
  $\Delta^{{\rm Chow}}_{V,F}$ of Lemma~\ref{lemma:deltachow}. When
  this is the case, we get that $C_V$ and
  $Q=\textsc{Degree\_$d$\_Chow\_form}(C_V, F)$ computed at Step 1 have
  coefficients in $\Z_{(p)}$, $Q \bmod
  p=\textsc{Degree\_$d$\_Chow\_form}(C_V \bmod p, F \bmod p)$, and
  $\gcd(C_V \bmod p,\partial C_V/\partial U_{0,0} \bmod p)=1$. In
  particular, this establishes the third item in the lemma.

  Let us further assume that $p$ does not divide the numerator
  $\Delta_1$ of the leading coefficient of $Q$. It follows the leading
  coefficient of $Q \bmod p$ is the leading coefficient of $Q$,
  reduced modulo $p$.
  We can then consider the absolute factorization of $Q$ as
  in~\eqref{eq:factoQ}:
  \[ Q = \lc(Q) \prod_{i \in I_1} C_i^{m_i} \prod_{i \in I_2} C_i^{m_i}.\]
  Because $\Q$ is perfect, $Q$ also factors over $\Q$ as
  \[ Q = A \prod_{j \in J} B_j^{\mu_j},\]
  with $A$ and all $B_j$ in $\Q[\bm U_0,\dots,\bm U_{r-1}]$, $A = \lc(Q)  \prod_{i \in I_1} C_i^{m_i}$ of
  degree $0$ in $U_{0,0}$, all $B_j$ monic and of positive degree in
  $U_{0,0}$, $\prod_{j \in J} B_j = \prod_{i \in I_2} C_i$ and
  $\{\mu_j\}_{j \in J}=\{m_i\}_{i \in I_2}$. 

  Let us write $Q=c Q'/e$, with $Q'$ primitive (thus, with integer
  coefficients) and $c,e$ coprime integers.  For $j$ in $J$, we can
  also rewrite $Q$ as $Q=R_j B_j$, with $R_j$ having coefficients in
  $\Q$, and thus $c Q'/e = r_j R'_j/s_j \ B'_j/t_j$, with $R'_j,B'_j$
  primitive and $r_j,s_j$ coprime ($B_j$ is monic, so $t_j B_j$ is
  primitive). This gives $c s_j t_j Q' = e r_j R'_j B'_j$, and so $c
  s_j t_j = e r_j$ by taking contents. On the other hand, since $B_j$
  is monic, the leading coefficient of $B'_j$ is $t_j$, so $c s_j t_j
  \lc(Q') = e r_j \lc(R'_j) t_j$, and $\lc(Q') = \lc(R'_j) t_j$.  Our
  choice of $p$ ensures that $\lc(Q') \bmod p$ is nonzero, so $B_j$
  has coefficients in $\Z_{(p)}$.

  On the other hand, since $Q=C_{\deg(F),V}[F]$
  (Lemma~\ref{lemma:degDCF}), we know that $C_W = \prod_{i \in I_2}
  C_i$, that is, $C_W=\prod_{j \in J} B_j$. It follows that $C_W$ has
  coefficients in $\Z_{(p)}$, so the first item in the lemma is
  established. Besides, we get from Lemma~\ref{lemma:divmodp} that it
  is also the case for $A$. Using lower case $a$ and $b_j$ for the
  reductions of the polynomials $A$ and $B_j$ modulo $p$, we get $Q
  \bmod p = a \prod_{j \in J} b_j^{\mu_j}$, with $a$ that does not
  depend on $U_{0,0}$. Since $C_W$ is normalized, all $C_i$'s are
  (Lemma~\ref{lemma:normal}), which implies that the leading terms of
  all $B_j$'s (for our graded order) have positive degree in
  $U_{0,0}$. Since the $B_j$'s are monic, the $b_j$'s are also all
  monic and of positive degree in $U_{0,0}$.

  Next, we give conditions that ensure the fourth item, that
  $\textsc{Degree\_$d$\_Chow\_form}(C_V \bmod p, F \bmod p)$ is tame.
  Assuming that $p$ does not divide $\Delta^{{\rm Chow}}_{V,F}
  \Delta_1$, we saw that this is equivalent to $Q \bmod p$ being tame.
  \begin{itemize}
  \item We let $\Delta_2$ be the product of all primes dividing one of
    the integers $\{m_i\}_{i\in I_2} =\{\mu_j\}_{j \in J}$
  \item We let $\Delta_3$ be a nonzero integer such that if $p$ does
    not divide $\Delta_3$, $\gcd(C_W \bmod p,\partial C_W/\partial
    U_{0,0} \bmod p)=1$. When this is case case, since $\prod_{j \in
      J} b_j=C_W \bmod p$, all $b_j$'s are pairwise coprime, and all
    $b_j$ are absolutely squarefree.
  \end{itemize}
  It follows that when $p$ does not divide $\Delta^{\rm
    intersection}_{V,F}=\Delta^{{\rm Chow}}_{V,F}
  \Delta_1\Delta_2\Delta_3$, $C_W \bmod p$ is the weak squarefree part
  of $Q \bmod p$, and the multiplicities of the corresponding factors
  in $Q \bmod p$ are all nonzero modulo $p$, so $Q \bmod p$ is
  tame. This gives the fourth property in the lemma.

  Finally, under the same assumption on $p$, we establish the second
  point in the lemma, that is, $\textsc{Intersection\_Chow}(C_V \bmod
  p, F \bmod p)=C_W \bmod p$. Since we proved that $Q \bmod
  p=\textsc{Degree\_$d$\_Chow\_form}(C_V \bmod p, F \bmod p)$, and
  that this polynomial is tame, we know (by means of
  Lemma~\ref{lemma:tame}) that $\textsc{Intersection\_Chow}(C_V \bmod
  p, F \bmod p)$ returns the weak squarefree part of $Q \bmod p$,
  which we saw is $C_W \bmod p$.
  
  At this stage, it remains to give a height bound for $\Delta^{\rm
    intersection}_{V,F}$. Lemma~\ref{lemma:deltachow} gives an upper
  bound of $O( n^6 D_V(D_V + H_V) + h_F)$ for the height of
  $\Delta^{{\rm Chow}}_{V,F}$. The height of $\Delta_1$ is less than
  or equal to $h(Q)=C_{\deg(F),V}[F]$, and we know from
  Lemma~\ref{lemma:hcdvf} that the latter is in $O( d_F H_V + h_F D_V
  + nd_FD_V\ln(n))$.

  For $\Delta_2$ and $\Delta_3$, we are going to use upper bounds on
  the degree and height of $W$: using B\'ezout's theorem and its
  arithmetic version~\cite[Corollary~2.10]{KrPaSo01}, they are
  respectively at most $D_W=d_F D_V$ and $H_W = d_F H_V + h_F D_V + n
  d_F D_V \ln(n+1)$. The second quantity, $\Delta_3$, is the simplest
  to deal with: Lemma~\ref{lemma:gcd2} shows that we can take
  $\Delta_3$ of height $O(n^6 D_W (D_W + H_W))$, which is $O(n^6 d_F
  D_V (d_F H_V + h_F D_V + n d_F D_V \ln(n))$.

  For $\Delta_2$, since $Q$ has degree at most $D_W$ in $U_{0,0}$, it
  follows that $h(\Delta_2)$ is bounded above by $\sum_{p {\rm~prime~}
    \le D_W} \ln(p)$. This sum is Dirichlet's theta function
  $\vartheta(D_W)$~\cite[p.~22]{BachSh1996}, which satisfies
  $\vartheta(D_W) \in O(D_W)=O(d_F D_V)$~\cite[Lemma~8.2.2]{BachSh1996}.
\end{proof}

  
\section{The algorithm}


\subsection{Geometric view}

In this section, we describe a (rather straightforward) procedure to compute
the equidimensional components of a set $V(F_1,\dots,F_s)$, for
polynomials $F_i$ with coefficients in a field $\K$, first in purely
geometric terms, then by means of their Chow forms.

Given $F_1,\dots,F_s$ in $\K[X_1,\dots,X_n]$, let us write $V_i =
V(F_1,\dots,F_i)$, with $V_0 =\Kbar^n$. For $i=0,\dots,s$, we write
the equidimensional decomposition of $V_i$ as
\[V_i=\bigcup_{0 \le k \le n} V_{i,k},\]
with $V_{i,k}$ $k$-equidimensional or empty for all $k$. In this notation,
and in the whole presentation below, we keep all indices $i,k,\dots$ explicit,
as this will make it convenient for us to discuss the reduction modulo a prime
$p$ of all steps in this process.

For $i=0,\dots,s-1$, we have
\[V_{i+1} = \bigcup_{0 \le k\le n} \left ( V_{i,k} \cap V(F_{i+1}) \right ).\]
For $i$ as above and $k=0,\dots,n$, we can write
\[V_{i,k} = V_{i,k,{\rm zero}} \cup V_{i,k,{\rm proper}}, \quad\text{with}\quad V_{i,k,{\rm zero}}, V_{i,k,{\rm proper}}= \textsc{Separate}(V_{i,k},V(F_{i+1})).\]
It follows that
$V_{i+1}$ can be written as the union of all $V_{i,k,{\rm zero}}$ (which are all $k$-equidimensional or empty) and all $V_{i,k,{\rm proper}}
\cap V(F_{i+1})$  (which are all $(k-1)$-equidimensional or empty). That is,
\[V_{i+1} = \bigcup_{0 \le k \le n} W_{i,k}, \quad\text{with}\quad W_{i,k} =  V_{i,k,{\rm zero}} \cup (V_{i,k+1,{\rm proper}} \cap V(F_{i+1}));\]
here, for $k=n$, we set $V_{i,n+1,{\rm proper}}=\emptyset$.
Each $W_{i,k}$ is $k$-equidimensional or empty, but the family
$(W_{i,k})_{0 \le k \le n}$ may not form the equidimensional
decomposition of $V_{i+1}$, since some irreducible components of a
given $W_{i,k}$ may be subsets of some $W_{i,k'}$, $k' > k$.

We rectify this by using the $\textsc{Separate}$ operation again. For
$k=0,\dots,n$ and $\ell=k+1,\dots,n+1$, let $Z_{i,k,\ell}$ be the
union of all irreducible components of $W_{i,k}$ that are not
contained in the union of $W_{i,\ell},\dots,W_{i,n}$. In particular,
we have $Z_{i,k,n+1}=W_{i,k}$, and  $V_{i+1,k}=Z_{i,k,k+1}$. Besides, it follows from the definition
that for $\ell=k+1,\dots,n$, we have the recurrence
\[Z_{i,k,\ell} = Z_{i,k,\ell+1,{\rm proper}},\]
with
\[Z_{i,k,\ell+1,{\rm zero}},Z_{i,k,\ell+1,{\rm proper}} = \textsc{Separate}(Z_{i,k,\ell+1}, W_{i,\ell}).\]
This leads us to the following first version of our algorithm.

\medskip\noindent{\sc Decomposition\_Geometric}\vspace{-1mm}
\begin{algorithmic}[1]
\Require $F_1,\dots,F_s$
\Ensure the equidimensional components $V_{s,0},\dots,V_{s,n}$ of $V(F_1,\dots,F_s)$
\State Set $V_{0,0} = \cdots = V_{0,n-1} = \emptyset$ and $V_{0,n} = \Kbar^n$
\For{$i=0,\dots,s-1$}
\State $V_{i,n+1,{\rm proper}} \gets \emptyset$
\For{$k=n,\dots,0$}
\State  $V_{i,k,{\rm zero}},V_{i,k,{\rm proper}} \gets \textsc{Separate}(V_{i,k}, V(F_{i+1}))$\label{step:sepF}
\State  $J_{i,k} \gets V_{i,k+1,{\rm proper}} \cap  V(F_{i+1})$\label{step:J}
\State  $W_{i,k} \gets V_{i,k,{\rm zero}} \cup  J_{i,k}$\label{step:W}
\State  $Z_{i,k,n+1} \gets W_{i,k}$\label{step:Z}
\For{$\ell=n,\dots,k+1$}
\State\label{step:sepW}  $\_,Z_{i,k,\ell} \gets \textsc{Separate}(Z_{i,k,\ell+1}, W_{i,\ell})$
\EndFor
\State  $V_{i+1,k} \gets Z_{i,k,k+1}$
\EndFor
\EndFor
\State \Return $V_{s,0},\dots,V_{s,n}$
\end{algorithmic}

In order to turn this into an actual algorithm, we will use the
operations on Chow forms described in the previous section.  Most
of them require normal position; the following properties will ensure
all conditions we need are satisfied.

\begin{definition}
  We say that $F_1,\dots,F_s$ are in {\em general position} if the
  following holds: for $i=0,\dots,s-1$ and $k=0,\dots,n$, $W_{i,k}$ is
  either empty or in normal position.
\end{definition}
This property holds in generic coordinates: we will return to this
discussion and quantify the bad changes of variables later on.

\begin{lemma}\label{lemma:genpos}
  For any input $F_1,\dots,F_s$, 
  \begin{enumerate}
  \item\label{it1a} $V_{0,0},\dots,V_{0,n}$ are all empty or in normal position.
  \end{enumerate}
  Besides, if $F_1,\dots,F_s$ are in general position, the following
  holds:
  \begin{enumerate}[start=2]
  \item for $i=0,\dots,s-1$, $k=0,\dots,n$ and $\ell=k+1,\dots,n+1$,
    all sets $V_{i,k,{\rm zero}}$, $V_{i,k,{\rm proper}}$, $J_{i,k}$,
    $W_{i,k}$, $Z_{i,k,\ell}$ (and thus $V_{i+1,k}$) are empty or 
    $k$-equidimensional and in normal position.
  \end{enumerate}
\end{lemma}
\begin{proof}
  First item is clear. For the second claim, all dimension statements
  are true by definition. That $W_{i,k}$ is in normal position follows
  from the general position assumption. Then, all sets $V_{i,k,{\rm
      zero}}$, $J_{i,k}$ and $Z_{i,k,\ell}$ consist of irreducible
  components of $W_{i,k}$ so they are in normal position (or empty).
  Similarly, if $i>0$, the set $V_{i,k}$ consists of irreducible components of
  $W_{i-1,k}$ and it is then empty or in normal position,
  and the same conclusion holds if $i=0$ by the first item; in any
  case, it is then also the case for $V_{i,k,{\rm proper}}$.
\end{proof}

We can then present a more concrete version of this algorithm, using
the operations on Chow forms seen so far. Since the subroutines it uses
may return {\nondef}, the whole procedure may do so as well.

\medskip\noindent{\sc Decomposition\_Chow}\vspace{-1mm}
\begin{algorithmic}[1]
  \Require $F_1,\dots,F_s$ 
  \Ensure polynomials $C_{s,0},\dots,C_{s,n}$ or \nondef
  \State \textbf{for} $i=1,\dots,n-1$  \textbf{do} $C_{0,i}  \gets 1$
  \State $C_{0,n} \gets \det(\bm U_0,\dots,\bm U_n)$, with $\bm U_i=U_{i,0},\dots,U_{i,n}$ for all $i$
  \For{$i=0,\dots,s-1$}
  \State $C_{i,n+1,{\rm proper}} \gets 1$
  \For{$k=n,\dots,0$}
  \State  $C_{i,k,{\rm zero}},C_{i,k,{\rm proper}} \gets \textsc{Separate\_Hypersurface\_Chow}(C_{i,k}, F_{i+1})$\label{step:CsepF2}
  \State  $D_{i,k} \gets \textsc{Intersection\_Chow}(C_{i,k+1,{\rm proper}}, F_{i+1})$\label{step:CJ2}
  \State  $R_{i,k} \gets \textsc{Union\_Chow}(C_{i,k,{\rm zero}}, D_{i,k})$\label{step:CW2}
  \State  $S_{i,k,n+1} \gets R_{i,k}$\label{step:CZ2}
  \For{$\ell=n,\dots,k+1$}
  \State\label{step:CsepW}  $\_,S_{i,k,\ell} \gets \textsc{Separate\_Chow}(S_{i,k,\ell+1}, R_{i,\ell})$
  \EndFor
  \State  $C_{i+1,k} \gets S_{i,k,k+1}$
  \EndFor
  \EndFor
  \State \Return $C_{s,0},\dots,C_{s,n}$
\end{algorithmic}

The following proposition is routine: we establish that when the input
is in general position, the previous procedure correctly computes the
Chow forms of the algebraic sets defined in the geometric version of
the algorithm. We make the assumption that $\K$ has characteristic zero,
as it ensures some assumptions (such as tameness) that are needed to
guarantee that the subroutines we use return the expected output.

\begin{proposition}\label{prop:welldef}
  For any input $F_1,\dots,F_s$, 
  \begin{enumerate}
  \item\label{it1} $C_{0,0},\dots,C_{0,n}$ are the normalized Chow forms of $V_{0,0},\dots,V_{0,n}$.
  \end{enumerate}
  Besides, if $F_1,\dots,F_s$ are in general position and $\K$ has characteristic zero, the algorithm
  does not return {\nondef} and
  \begin{enumerate}[start=2]
  \item\label{it2} for $i=0,\dots,s-1$, $k=0,\dots,n$ and $\ell=k+1,\dots,n+1$,
    \begin{enumerate}
    \item\label{ita} $C_{i,n+1,{\rm proper}}$ is the normalized Chow form of $V_{i,n+1,{\rm proper}}$
    \item\label{itb} $C_{i,k,{\rm zero}}$ is the normalized Chow form of $V_{i,k,{\rm zero}}$
    \item\label{itc} $C_{i,k,{\rm proper}}$ is the normalized Chow form of $V_{i,k,{\rm proper}}$
    \item\label{itd} $D_{i,k}$ is the normalized Chow form of $J_{i,k}$
    \item\label{ite} $R_{i,k}$ is the normalized Chow form of $W_{i,k}$
    \item\label{itf} $S_{i,k,\ell}$ is the normalized Chow form of $Z_{i,k,\ell}$
    \item\label{itg} $C_{i+1,k}$ is the normalized Chow form of $V_{i+1,k}$
    \end{enumerate}
  \end{enumerate}
\end{proposition}
\begin{proof}
  We prove by induction on $i$ that for $i=-1,\dots,s-1$,
  claim~\ref{itg} holds, that is, $C_{i+1,k}$ is well-defined and is
  the normalized Chow form of $V_{i+1,k}$ for all $k=0,\dots,n$. We
  establish all other properties as part of the induction.

  For $i=-1$, we have to prove that $C_{0,0},\dots,C_{0,n}$ are the
  normalized Chow forms of $V_{0,0},\dots,V_{0,n}$; this is obvious by
  inspection, with or without the general position assumption, and
  proves claim~\ref{it1}. Then, for some $i=0,\dots,s-1$, we assume
  that $C_{i,k}$ is the normalized Chow form of $V_{i,k}$ for all
  $k=0,\dots,n$, and we prove that it remains true for $C_{i+1,k}$,
  for all $k$ as well.
  \begin{itemize}
  \item $V_{i,n+1,{\rm proper}}$ is empty, and accordingly we set
    $C_{i,n+1,{\rm proper}}=1$, which gives \ref{ita}.
  \item For $k=0,\dots,n$, because $V_{i,k}$ is either empty or in
    normal position (Lemma~\ref{lemma:genpos}), and using both our
    induction assumption for $i-1$ and the fact that we are in
    characteristic zero, Lemmas~\ref{lemma:gcdeq1}
    and~\ref{lemma:separate} gives claims \ref{itb} and \ref{itc}.
  \item For $k=0,\dots,n$, $V_{i,k+1,{\rm proper}}$ is either empty or
    $(k+1)$-equidimensional and in normal position (by definition if
    $k=n$, Lemma~\ref{lemma:genpos} otherwise), $C_{i,k+1,{\rm
      proper}}$ is its normalized Chow form (claim \ref{ita} if $k=n$,
    \ref{itc} otherwise), and $J_{i,k}$ is either empty or
    $k$-equidimensional and in normal position
    (Lemma~\ref{lemma:genpos}). Using the characteristic zero
    assumption allows us to verify the last two assumptions in
    Lemma~\ref{lemma:interchow} (the former from
    Lemma~\ref{lemma:gcdeq1}, the latter from the definition of
    tameness in Subsection~\ref{ssec:propint}). Then,
    Lemma~\ref{lemma:interchow} gives claim \ref{itd}, that $D_{i,k}$ is the normalized Chow form of $J_{i,k}$.
  \item For $k=0,\dots,n$, both $V_{i,k,{\rm zero}}$ and $J_{i,k}$ are
    either empty or $k$-equidimensional and in normal position
    (Lemma~\ref{lemma:genpos}). Using items \ref{itb} and \ref{itd},
    and the characteristic zero assumption (for
    Lemma~\ref{lemma:gcdeq1}, and because characteristic two is
    explicitly excluded), Lemma~\ref{lemma:unionchow} gives claim
    \ref{ite}, that $R_{i,k}$ is the normalized Chow form of $W_{i,k}$.
  \item For $k=0,\dots,n$, we establish \ref{itf} ($S_{i,k,\ell}$ is
    the normalized Chow form of $Z_{i,k,\ell}$) by decreasing
    induction on $\ell=n+1,\dots,k+1$. For $\ell=n+1$, we already know
    that $S_{i,k,n+1}=R_{i,k}$ is the normalized Chow form of
    $W_{i,k}=Z_{i,k,n+1}$ by~\ref{ite}. Assume that, for some $\ell
    \in \{k+1,\dots,n\}$, we know that $S_{i,k,\ell+1}$ is the
    normalized Chow form of $Z_{i,k,n+1}$. We also know that
    $R_{i,\ell}$ is the normalized Chow form of $W_{i,\ell}$
    (\ref{ite} again), so, being in characteristic zero,
    Lemma~\ref{lemma:separategeneral} shows that $S_{i,k,\ell}$ is the
    normalized Chow form of $Z_{i,k,\ell}$.  This proves \ref{itf} for
    all values of $\ell$.  Taking $\ell=k+1$, we get
    \ref{itg}. \qedhere
  \end{itemize}
\end{proof}


\subsection{Reduction modulo $p$}

The polynomials $F_1,\dots,F_s$ are as above, and we now take $\K=\Q$
and the $F_i$'s with integer coefficients; in this section, we still
assume that these polynomials are in general position.

Fix a prime $p$ and let $f_1,\dots,f_s=F_1,\dots,F_s \bmod p$.
Suppose we run procedure {\sc Decomposition\_Geometric} on input
$f_1,\dots,f_s$: we denote the algebraic sets it computes (in
$\overline \F_p^n$) using lower case letters, such as $v_{i,k}$,
$w_{i,k}$, $z_{i,k,\ell}$, etc. Because {\sc Decomposition\_Geometric} makes
no assumption on its input, we know that $v_{s,0},\dots,v_{s,n}$ are
the equidimensional components of $V(f_1,\dots,f_s)$.

On the other hand, we do not assume that $f_1,\dots,f_s$ are in
general position, so we can {\it a priori} make no statement about
what {\sc Decomposition\_Chow} returns on input $f_1,\dots,f_s$
(the output may be undefined). As we did for {\sc
  Decomposition\_Geometric}, we will denote using lower case letters,
such as $c_{i,k}$, $d_{i,k}$, etc, the polynomials with coefficients
in $\F_p$ computed by {\sc Decomposition\_Chow} on input
$f_1,\dots,f_s$, provided they are not undefined.

The following proposition shows that for all primes except a finite
number, reducing modulo $p$ the Chow forms of the equidimensional
components of $V(F_1,\dots,F_s)$ gives the Chow forms of the
equidimensional components of $V(f_1,\dots,f_s)$.
\begin{proposition}\label{prop:Delta}
  Suppose that $F_1,\dots,F_s$ are in general position. For $i=0,\dots,s-1$, the integers
  \[\Delta_i = \prod_{k=0}^n \left (\Delta^{\rm separate\_H}_{V_{i,k},F_{i+1}} \Delta^{\rm intersection}_{V_{i,k+1,{\rm proper}},F_{i+1}}
  \Delta^{\rm union}_{V_{i,k,{\rm zero}},J_{i,k}} \prod_{\ell=k+1}^{n} \Delta^{\rm separate}_{Z_{i,k,\ell+1},W_{i,\ell}} \right )\]
  as well as
  \[ \Delta = 2 \prod_{i=0}^{s-1} \Delta_i\]
  are well-defined and nonzero. If $p$ does not divide $\Delta$, then
  $C_{s,0},\dots,C_{s,n}$ have coefficients in $\Z_{(p)}$,
  $c_{s,0},\dots,c_{s,n}$ are well-defined and are their reductions modulo $p$,
  $v_{s,0},\dots,v_{s,n}$ are all either empty or in normal position, 
  and $c_{s,0},\dots,c_{s,n}$ are their normalized Chow forms.
\end{proposition}
The rest of this section gives the proof of this proposition.
We first establish that the integers $\Delta_i$ are well-defined and
nonzero. For $i=0,\dots,s-1$, we have the following:
\begin{itemize}
\item for $k=0,\dots,n$, $V_{i,k}$ is either empty or equidimensional
  and in normal position by Lemma~\ref{lemma:genpos}, so the factor
  $\Delta^{\rm separate\_H}_{V_{i,k},F_{i+1}}$ is well-defined and nonzero by
  Lemma~\ref{lemma:HsepH}.
\item for $k=0,\dots,n$, $V_{i,k+1,{\rm proper}}$ is either empty or
  equidimensional and in normal position by Lemma~\ref{lemma:genpos}
  again. Besides, the intersection $V_{i,k+1,{\rm proper}}$ and
  $V(F_{i+1})$ is proper, so $\Delta^{\rm intersection}_{V_{i,k+1,{\rm
        proper}},F_{i+1}}$ is well-defined and nonzero by
  Lemma~\ref{lemma:interchowH}.
\item for $k=0,\dots,n$, $V_{i,k,{\rm zero}}$ and $J_{i,k}$ are either
  empty or $k$-equidimensional and in normal position by
  Lemma~\ref{lemma:genpos}, so by Lemma~\ref{lemma:unionH}, 
  $\Delta^{\rm union}_{V_{i,k,{\rm zero}},J_{i,k}}$ is well-defined
  and nonzero.
\item for $k=0,\dots,n$ and $\ell=k+1,\dots,n$, $Z_{i,k,\ell+1}$ and
  $W_{i,\ell}$ are either empty or equidimensional and in normal
  position, still by Lemma~\ref{lemma:genpos}. By
  Lemma~\ref{lemma:Hsep}, $\Delta^{\rm
    separate}_{Z_{i,k,\ell+1},W_{i,\ell}}$ is well-defined and
  nonzero.
\end{itemize}

It follows that all $\Delta_i$, and thus $\Delta$, are nonzero
integers. Next, in the sequel, we will mention several times a number of
related properties, for which we introduce the following notation. Let
$c$ be a polynomial in $\F_p[\bm U_0,\dots,\bm U_r]$, with $U_i =
U_{i,0},\dots,U_{i,n}$ for all $i$, let $C$ be in $\Q[\bm
  U_0,\dots,\bm U_r]$, and let $v$ be an algebraic set in
$\overline{\F_p}^n$. We say that $c,C,v$ have property $\sfP$ if
\begin{itemize}
\item[$\sfP_1$:] $C$ has coefficients in $\Z_{(p)}$
\item[$\sfP_2$:] $c = C \bmod p$
\item[$\sfP_3$:] $v$ is either empty or $r$-equidimensional and in normal position
\item[$\sfP_4$:] $c$ is the normalized Chow form of $v$.
\end{itemize}
Then, for $i=0,\dots,s$ and $k=0,\dots,n$, the key property we want to establish
is the following:
\begin{description}
\item[$\sfA(i,k):$] $c_{i,k}$ is well-defined and $(c_{i,k},C_{i,k},v_{i,k})$
  have property $\sfP$.
\end{description}
We prove by increasing induction on $i=0,\dots,s$ that, upon imposing
suitable divisibility restrictions on $p$, $\sfA(i,k)$ holds for all
values of $k$. Then, taking $i=s$, our conditions on $p$ will amount
to $\Delta \bmod p \ne 0$, and the resulting properties will prove the
proposition, since $\sfA(s,k)$ says that for all $k$, $C_{s,k}$ has coefficients in
$\Z_{(p)}$, $c_{s,k}$ is well-defined, $c_{s,k} = C_{s,k} \bmod p$,
$v_{s,k}$ is either empty or in normal position and $c_{s,k}$ is its
normalized Chow form.

For $i=0$, $\sfA(0,k)$ holds for all values of $k$: $\sfP_1$ and
$\sfP_2$ hold by inspection, $\sfP_3$ is from item~\ref{it1a} of
Lemma~\ref{lemma:genpos} and $\sfP_4$ is item~\ref{it1} in
Proposition~\ref{prop:welldef} (applied to $f_1,\dots,f_s$ in both
cases).

Suppose now that for some $i$ in $0,\dots,s-1$ we have established all
$\sfA(i,k)$, $k=0,\dots,n$. Assume from now on that $p$ does not
divide $2\Delta_i$, and let us prove that all $\sfA(i+1,k)$ hold, for
$k=0,\dots,n$.
\begin{enumerate}
\item At Step 4, $v_{i,n+1,{\rm proper}}$ is empty, $c_{i,n+1,{\rm
    proper}}=1$ (as a polynomial with coefficients in $\F_p$) and
  $C_{i,n+1,{\rm proper}}=1$ (as a polynomial with coefficients in
  $\Q$), so the triple $(c_{i,n+1,{\rm proper}},C_{i,n+1,{\rm
    proper}},v_{i,n+1,{\rm proper}})$ has property $\sfP$.
\item Take $k$ in $0,\dots,n$. Over $\F_p$, we have
  \begin{equation}\label{eq:vikz}
    v_{i,k,{\rm zero}},v_{i,k,{\rm proper}} = \textsc{Separate}(v_{i,k}, V(f_{i+1}))    
  \end{equation}
  \[c_{i,k,{\rm zero}},c_{i,k,{\rm proper}} = \textsc{Separate\_Hypersurface\_Chow}(c_{i,k}, f_{i+1})\]
(unless the function call returns \nondef), whereas
  and the following holds over $\Q$:
  \[C_{i,k,{\rm zero}},C_{i,k,{\rm proper}} = \textsc{Separate\_Hypersurface\_Chow}(C_{i,k}, F_{i+1}).\]
  By Proposition~\ref{prop:welldef} we have $C_{i,k}=C_{V_{i,k}}$, and
  by the first and second items of $\sfA(i,k)$, it has coefficients in
  $\Z_{(p)}$, with $c_{i,k} = C_{i,k} \bmod p$.  Then, using the definitions above, and because $p$
  does not divide $\Delta^{\rm separate\_H}_{V_{i,k},F_{i+1}}$, we
  obtain the following properties from Lemma~\ref{lemma:HsepH}:
  \begin{itemize}
  \item $C_{i,k,{\rm zero}}$ and $C_{i,k,{\rm proper}}$ have coefficients in $\Z_{(p)}$.
  \item $c_{i,k,{\rm zero}},c_{i,k,{\rm proper}}$ are well-defined and equal to $C_{i,k,{\rm zero}} \bmod p,C_{i,k,{\rm proper}} \bmod p$.
  \item $\gcd(c_{i,k}, \partial c_{i,j}/\partial U_{0,0})=1$.
  \end{itemize}
  On the other hand, since $v_{i,k}$ is either empty or in normal
  position and $c_{i,k}$ is its normalized Chow form (both also by
  $\sfA(i,k)$), and using the third item listed above,
  Lemma~\ref{lemma:separate} shows that $c_{i,k,{\rm zero}}$ and
  $c_{i,k,{\rm proper}}$ are the respective normalized Chow forms of
  $v_{i,k,{\rm zero}}$ and $v_{i,k,{\rm proper}}$.

  This proves that $(c_{i,k,{\rm proper}},C_{i,k,{\rm
      proper}},v_{i,k,{\rm proper}})$ and $(c_{i,k,{\rm
      zero}},C_{i,k,{\rm zero}},v_{i,k,{\rm zero}})$ have property
  $\sfP$, for all $k$ in $0,\dots,n$.
\item Take $k$ in $0,\dots,n$. Over $\F_p$ we have 
  \[j_{i,k} = v_{i,k+1,{\rm proper}} \cap  V(f_{i+1})\]
  \[d_{i,k} = \textsc{Intersection\_Chow}(c_{i,k+1,{\rm proper}}, f_{i+1})\]
  (unless the function call returns \nondef), whereas over $\Q$, the following holds:
  \[D_{i,k} = \textsc{Intersection\_Chow}(C_{i,k+1,{\rm proper}}, F_{i+1}).\]

  Recall that $C_{i,k+1,{\rm proper}}$ and $D_{i,k}$ are the
  normalized Chow forms of $V_{i,k+1,{\rm proper}}$ and
  $J_{i,k}=V_{i,k+1,{\rm proper}} \cap V(F_{i+1})$
  (Proposition~\ref{prop:welldef}). We also know that $C_{i,k+1,{\rm
      proper}}$ has coefficients in $\Z_{(p)}$, with $c_{i,k+1,{\rm
      proper}}=C_{i,k+1,{\rm proper}} \bmod p$ (item $1$ in this proof if $k=n$, item 2 otherwise).  Then, since
  we assume that $p$ does not divide $\Delta^{\rm
    intersection}_{V_{i,k+1,{\rm proper}},F_{i+1}}$, we obtain the
  following from Lemma~\ref{lemma:interchowH}: 
  \begin{itemize}
  \item $D_{i,k}$ has coefficients in $\Z_{(p)}$.
  \item $d_{i,k}$ is well-defined and equal to $D_{i,k} \bmod p$.
  \item $\gcd(c_{i,k+1,{\rm proper}},\partial c_{i,k+1,{\rm proper}}/\partial U_{0,0})=1$.
  \item $\textsc{Degree\_$d$\_Chow\_form}(c_{i,k+1,{\rm proper}}, f_{i+1})$ is tame.
  \end{itemize}
  On the other hand, we know that $v_{i,k+1,{\rm proper}}$ is empty or
  in normal position (again, item 1 if $k=n$ and 2 otherwise), that
  $v_{i,k+1,{\rm proper}} \cap V(f_{i+1})$ is empty or
  $k$-equidimensional (by item 1 if $k=n$, and Eq.~\eqref{eq:vikz} at index $k+1$ otherwise), and that
  $c_{i,k+1,{\rm proper}}$ is the normalized Chow form of
  $v_{i,k+1,{\rm proper}}$ (items 1 and 2 again). Using the third and
  fourth items listed above, Lemma~\ref{lemma:interchow} shows
  that $d_{i,k}$ is a Chow form of $j_{i,k}$.
  
  We know that $D_{i,k}$ is normalized, so its modulo $p$ reduction
  $d_{i,k}$ is as well. This means that $j_{i,k}$, if not empty, is in
  normal position and that $d_{i,k}$ is its normalized Chow form. In
  other words, $(d_{i,k},D_{i,k},j_{i,k})$ have property $\sfP$.

\item Take $k$ in $0,\dots,n$. Over $\F_p$ we have
  \begin{equation*}
    w_{i,k} = v_{i,k,{\rm zero}} \cup j_{i,k}    
  \end{equation*}
  \begin{equation*}
    r_{i,k} = \textsc{Union\_Chow}(c_{i,k,{\rm zero}}, d_{i,k})
  \end{equation*}
  (unless the function call returns \nondef), whereas over $\Q$ we have
  \begin{equation*}
    R_{i,k} = \textsc{Union\_Chow}(C_{i,k,{\rm zero}}, D_{i,k}).
  \end{equation*}
  Recall that $C_{i,k,{\rm zero}}$ is the normalized Chow form of
  $V_{i,k,{\rm zero}}$ and that $D_{i,k}$ is the normalized Chow form
  of $J_{i,k}$ (Proposition~\ref{prop:welldef}). We also know that
  $C_{i,k,{\rm zero}}$ has coefficients in $\Z_{(p)}$, with
  $c_{i,k,{\rm zero}}=C_{i,k,{\rm zero}} \bmod p$ (item 2 in this
  proof), and similarly that $D_{i,k,{\rm zero}}$ has coefficients in
  $\Z_{(p)}$, with $d_{i,k,{\rm zero}}=D_{i,k,{\rm zero}} \bmod p$
  (item 3 in this proof).  Then, since we assume that $p$ does not
  divide $\Delta^{\rm union}_{V_{i,k,{\rm zero}},J_{i,k}}$, we obtain
  the following properties from Lemma~\ref{lemma:unionH}:
  \begin{itemize}
  \item $R_{i,k}$ has coefficients in $\Z_{(p)}$.
  \item $r_{i,k}$ is well-defined and equal to $R_{i,k} \bmod p$.
  \item $\gcd(c_{i,k,{\rm zero}},\partial c_{i,k,{\rm zero}}/\partial U_{0,0})
         =\gcd(d_{i,k},\partial d_{i,k}/\partial U_{0,0})=1$.
  \end{itemize}
  On the other hand, we know that $v_{i,k,{\rm zero}}$ is empty or in
  normal position and that $c_{i,k,{\rm zero}}$ is its normalized Chow
  form (item 2), and similarly that $j_{i,k}$ is empty or in normal
  position and that $d_{i,k}$ is its normalized Chow form (previous
  item), both sets having dimension $k$ (if not empty). This in particular shows that
  $w_{i,k}$ is empty or in normal position. Since $p$ is
  odd, and using the third item listed above,
  Lemma~\ref{lemma:unionchow} shows that $r_{i,k}$ is the normalized Chow form of $w_{i,k}$.
  In other words, $(r_{i,k},R_{i,k},w_{i,k})$ have property~$\sfP$.

\item Take again $k$ in $0,\dots,n$. We have $s_{i,k,n+1}=r_{i,k}$,
  $S_{i,k,n+1}=R_{i,k}$ and $z_{i,k,n+1}=w_{i,k}$. Because
  $(r_{i,k},R_{i,k},w_{i,j})$ have property $\sfP$, we deduce that
  $(s_{i,k,n+1},S_{i,k,n+1},z_{i,k,n+1})$ do as well. Then, for
  $\ell=n,\dots,k+1$, assuming that
  $(s_{i,k,\ell+1},S_{i,k,\ell+1},z_{i,k,\ell+1})$ have property
  $\sfP$, we prove that $(s_{i,k,\ell},S_{i,k,\ell},z_{i,k,\ell})$ do
  as well.

  Over $\F_p$, we have
  \begin{equation*}
  \_,z_{i,k,\ell} \gets \textsc{Separate}(z_{i,k,\ell+1}, w_{i,\ell}),
  \end{equation*}
  and
  \begin{equation*}
  \_,s_{i,k,\ell} \gets \textsc{Separate\_Chow}(s_{i,k,\ell+1}, r_{i,\ell})
  \end{equation*}
  (unless the function call returns \nondef), whereas over $\Q$, the following holds:
  \begin{equation*}
  \_,S_{i,k,\ell} \gets \textsc{Separate\_Chow}(S_{i,k,\ell+1}, R_{i,\ell}).
  \end{equation*}
  Recall that $S_{i,k,\ell+1}$ and $R_{i,\ell}$ are the normalized
  Chow forms of $Z_{i,k,\ell+1}$ and $W_{i,\ell}$
  (Proposition~\ref{prop:welldef}), both polynomials have coefficients
  in $\Z_{(p)}$ (induction assumption for the former, item 4 for the
  latter) and their reductions modulo $p$ are $s_{i,k,\ell+1}$ and
  $r_{i,\ell}$ (same reasons). Then, since we assume that $p$ does not
  divide $\Delta^{\rm separate}_{Z_{i,k,\ell+1},W_{i,\ell}}$, we obtain the following
  properties from Lemma~\ref{lemma:Hsep}:
  \begin{itemize}
  \item $S_{i,k,\ell}$ has coefficients in $\Z_{(p)}$.
  \item $s_{i,k,\ell}$ is well-defined and equal to $S_{i,k,\ell} \bmod p$.
  \item $\gcd(s_{i,k,\ell+1}, \partial s_{i,k,\ell+1}/\partial U_{0,0})=1$.
  \end{itemize}
  On the other hand, we know that both $z_{i,k,\ell+1}$ and
  $w_{i,\ell}$ are empty or in normal position, by induction
  assumption for the former and item 4 for the latter, and that
  $s_{i,k,\ell+1}$ and $r_{i,\ell}$ are their normalized Chow forms
  (same reasons). Using the third item listed above,
  Lemma~\ref{lemma:separategeneral} shows that $z_{i,k,\ell}$ is either
  empty or in normal position and that $s_{i,k,\ell}$ is its
  normalized Chow form.

  In other words, $(s_{i,k,\ell},S_{i,k,\ell},z_{i,k,\ell})$ have
  property $\sfP$. This establishes our induction property, and taking
  $\ell=k+1$, this shows that $(c_{i+1,k},C_{i+1,k},v_{i+1,k})$ have
  property $\sfP$. In turn, this establishes $\sfA(i+1,k)$ and
  concludes the proof.
\end{enumerate}


\subsection{Quantitative aspects and proof of the main theorem}

We end this section with the proof of our main theorem. Let then
$F_1\dots,F_s$ be given, with integer coefficients and of degree at
most $d \ge 1$ and height at most $h \ge 1$. We do not assume that
they are in general position. As before, we call
$V_{i,k},J_{i,k},W_{i,k},\dots$ the algebraic sets that are obtained
by applying Algorithm {\sc Decomposition\_Geometric} to
$F_1,\dots,F_s$.

\begin{lemma}\label{lemma:DH}
  Suppose that all $F_i$'s have integer coefficients, degree at most
  $d \ge 1$ and height at most $h \ge 1$. Then, for $i=0,\dots,s-1$
  and $k=0,\dots,n$, all algebraic sets $V_{i,k}$, $V_{i,k,{\rm
      zero}}$, $V_{i,k,{\rm proper}}$, $V_{i,n+1,{\rm proper}}$,
  $J_{i,k}$, $W_{i,k}$ and $V_{i+1,k}=Z_{i,k,k+1},\dots,Z_{i,k,n+1}$ have degree
  at most $D$ and height at most $H$, for $D$ and $H$ given by
  \[D=2 d^{n+1} \in O(d^{n+1})\] and
  \[H= 5 h n d^{n+1} \ln(n+1)\in O(n h d^{n+1}\ln(n)).\]
\end{lemma}
\begin{proof}
  For $i=0,\dots,s$ and $k=0,\dots,n$, $V_{i,k}$ has degree at most
  $d^n$ and height at most $d^n (nh + 2n\ln(n+1))$: the degree bound
  is from~\cite[Proposition~2.3]{HeSc80}, and the height bound
  from~\cite[Corollary~2.10]{KrPaSo01}. If we consider only
  $i=0,\dots,s-1$, the same bounds hold for $V_{i,k,{\rm zero}}$ and
  $V_{i,k,{\rm proper}}$, since their irreducible components are
  irreducible components of $V_{i,k}$, as well as for
  $V_{i,n+1,{\rm proper}}=\emptyset$. From this, we get that the
  degree of $J_{i,k}=V_{i,k+1,{\rm proper}} \cap V(F_{i+1})$ is at
  most $d^{n+1}$, and that its height is at most $d^n h + d^{n+1} (nh
  + 2n\ln(n+1)) + d^{n+1} \ln(n+1)$ (same references). Finally, by
  additivity, the degree and height of $W_{i,k} = V_{i,k,{\rm zero}}
  \cup J_{i,k}$ are at most respectively
  \[D=2 d^{n+1} \in O(d^{n+1})\] and
  \[H= 5h n d^{n+1}  \ln(n+1)\in O(n h d^{n+1}\ln(n)).\]
  For $\ell=k+1,\dots,n$, the algebraic sets $Z_{i,k,\ell+1}$ admit
  the same degree and height bounds, since their irreducible
  components are irreducible components of $W_{i,k}$.
\end{proof}

In order to apply the results of the previous section, we need
polynomials in general position; this is going to be achieved by
applying a change of variables with integer coefficients of controlled
height (see~\cite[Section~4]{KrPaSo01} for arguments that inspired
this discussion, in the context of Nullstellensatz identities).

For an invertible $n\times n$ matrix $\mA$ and a polynomial $F$ in
$\Q[X_1,\dots,X_n]$, let us write $F^\mA=F(\mA \mX)$, where $\mX$ is
the column-vector with entries $X_1,\dots,X_n$. If $V$ is an algebraic
set, $V^\mA$ is defined as the zero-set of all polynomials $F^\mA$,
for $F$ in $I(V)$; equivalently, $V^\mA$ is the image of $V$ under the
map $\vx \mapsto \mA^{-1} \vx$. In particular, the algebraic sets
obtained by applying Algorithm {\sc Decomposition\_Geometric} to
$F_1^\mA,\dots,F_s^\mA$ are precisely
$V_{i,k}^\mA,J_{i,k}^\mA,W_{i,k}^\mA$, etc.

\begin{lemma}
  There exists $\mB \in \Z^{n \times n}$ such that the following holds:
  \begin{itemize}
  \item $\mB$ has height $O(n \ln(d))$.
  \item $\mB$ is invertible in $\Q^{n \times n}$, with inverse called $\mA$.
  \item $F_1^\mA,\dots,F_s^\mA$ are in general position.
  \end{itemize}
\end{lemma}
\begin{proof}
  For $i=0,\dots,s-1$ and $k=0,\dots,n$, $W_{i,k}$ is
  $k$-equidimensional (or empty), with degree at most $D =
  2d^{n+1}$. We can assume $k \ge 1$, since a zero-dimensional set is
  always in normal position.

  If $W_{i,k}$ is empty, the only constraint is that $\mB$ be
  invertible (since then, with $\mA=\mB^{-1}$, $W_{i,k}^\mA$ is also
  empty). In this case, we define $G_{i,k}=1$.

  Else, if $W_{i,k}$ is nonempty, Proposition~4.5 in~\cite{KrPaSo01}
  proves that there exists a nonzero polynomial $G_{i,k}$ in $\Q[\bm
    U_1,\dots,\bm U_k]$ (with $\bm U_i=U_{i,0},\dots,U_{i,n}$ for all
  $i$), of total degree at most $2nD^2$, and such that if
  $G_{i,k}(b_{1,0},\dots,b_{k,n})\ne 0$,
  \[W_{i,k} \cap V(b_{1,0} + b_{1,1}X_1 + \cdots + b_{1,n} X_n,\dots,b_{k,0} + b_{k,1}X_1 + \cdots + b_{k,n} X_n)\]
  has cardinality $\deg(W_{i,k})$.

  We then choose $b_{1,0},\dots,b_{n,n}$ integers such that
  $G_{i,k}(b_{1,0},\dots,b_{k,n})\ne 0$ for all $k=1,\dots,n$, and
  such that the matrix $\mB=[b_{k,j}]_{1 \le k,j \le n}$ is
  invertible. If we set $\mA=\mB^{-1}$, then for all $j=1,\dots,n$,
  $V(b_{j,0} + b_{j,1}X_1 + \cdots + b_{j,n} X_n)^\mA=V(X_j+b_{j,0})$.
  It follows that for $k=1,\dots,n$,
  \begin{align*}
  W_{i,k}^\mA \cap V(b_{1,0} + b_{1,1}X_1 + \cdots + b_{1,n} X_n,\dots,b_{k,0} + b_{k,1}X_1 + \cdots + b_{k,n} X_n)^\mA \\
  = W_{i,k}^\mA \cap V(X_1 +b_{1,0},\dots,X_k+b_{k,0}).
  \end{align*}
  Our assumption on the $b_{k,j}$'s shows that (if $W_{i,k}$ is not
  empty) this set is finite of degree
  $\deg(W_{i,k})=\deg(W^\mA_{i,k})$, so by Lemma~\ref{lemma:normaleq},
  $W_{i,k}^\mA$ is in normal position.

  It remains to give an upper bound on the height of $\mB$. The
  constraints on $\mB$ show that the $b_{k,j}$'s should not cancel a
  nonzero polynomial of degree at most $2n^2D^2 + n \le 3n^2D^2 \le
  12n^2 d^{2n+2}$, so by the De Millo-Lipton-Schwartz-Zippel Lemma, we
  can find suitable $b_{k,j}$'s with height at most $\ln(12n^2
  d^{2n+2}) \in O(n \ln(d))$.
\end{proof}

Fix $\mB$ and $\mA$ as in the previous lemma, and let $\delta$ be the
determinant of $\mB$; this is a nonzero integer of height $O(n^2
\ln(d))$, and we can write $\mA = \mA'/\delta$, where the entries of
$\mA'$ are also integers of height $O(n^2 \ln(d))$.

Define $G_1,\dots,G_s$ by $G_i = \delta_i F^\mA_i$, where $\delta_i$
is a minimal denominator for $F^\mA_i$ (this is a factor of some power
of $\delta$). It follows that $G_1,\dots,G_s$ are polynomials in
general position in $\Z[X_1,\dots,X_n]$, with degree at most $d$ and
height in $O(h')$, with $h'=h+ n^2d\ln(d)$; the latter is seen by
homogenizing the $F_i$'s, evaluating at
$(\mA'\mX)_1,\dots,(\mA'\mX)_n,\delta$ and
applying~\cite[Lemma~1.2.1]{KrPaSo01}.

Note also that the algebraic sets obtained by applying Algorithm {\sc
  Decomposition\_Geome\-tric} to $G_1,\dots,G_s$ are still
$V_{i,k}^\mA,J_{i,k}^\mA,W_{i,k}^\mA,\dots$ Since all
$V_{i,k},J_{i,k},W_{i,k},\dots$ have degree at most $D$ and height at
most $H$ as defined in Lemma~\ref{lemma:DH}, and all entries of
$\mB=\mA^{-1}$ have height $O(n\ln(d))$, Lemma~2.6 in~\cite{KrPaSo01}
shows that all $V_{i,k}^\mA,J_{i,k}^\mA,W_{i,k}^\mA,\dots$ have height
$O(H')$, with $H'=H + n^2 D \ln(d)$. Their degrees remain bounded
above by $D$.

Let $\Delta$ be the nonzero integer obtained by applying
Proposition~\ref{prop:Delta} to $G_1,\dots,G_s$, and let
$\Delta'=\delta\Delta$.  From Lemmas~\ref{lemma:HsepH},
\ref{lemma:interchowH}, \ref{lemma:unionH} and~\ref{lemma:Hsep}, we
get the respective height bounds:
\begin{align*}
 h(\Delta^{\rm separate\_H}_{V^\mA_{i,k},G_{i+1}}) & \in O(n^6 D ( d\ln(d)  + h' +  dH' + d D))\\
 h(\Delta^{\rm intersection}_{V^\mA_{i,k+1,{\rm proper}},G_{i+1}}) &\in O(n^6 d^2 D H' + n^6 d h' D^2  + n^7 d^2  D^2\ln(n))\\
 h(\Delta^{\rm union}_{V^\mA_{i,k,{\rm zero}},J^\mA_{i,k}}) &\in O(n^6 D(H' + D)) \\
 h(\Delta^{\rm separate}_{Z^\mA_{i,k,\ell+1},W^\mA_{i,\ell}})& \in O( n^6 D^2 (H' +D)).
\end{align*}
There are $s$ values of the index $i$ to consider, $n$ values for $k$
and at most $n$ for $\ell$.  To get an overall upper bound on the
height of $\Delta$, we take all of them into account, and use
the inequalities $D \ge d$, $H' \ge h$, $H'\ge h'$. This gives
\[ h(\Delta) \in O( n^{10}s ( d^2 DH'+ dh' D^2 + d^2 D^2 + D^3 + D^2H' )).\]
Using the upper bounds for $D$ and $H$, we obtain
\[ h(\Delta) \in O( n^{13} s h d^{3n+3} \ln(nd)) \subset O( n^{14} s h d^{3n+4}).\]
The height of $\Delta'$ admits the same asymptotic bound.

Let finally $p$ be a prime that does not divide $\Delta'$, and let us
call $\Phi_{0},\dots,\Phi_{n}$ the normalized Chow forms of the
equidimensional components of
$V(G_1,\dots,G_s)=V(F_1,\dots,F_s)^\mA$. Then, it follows from
Proposition~\ref{prop:Delta} that if we set $g_i = G_i \bmod p$ for
all $i$, $\Phi_{0},\dots,\Phi_{n}$ all have coefficients in
$\Z_{(p)}$, the equidimensional components of $V(g_1,\dots,g_s)$ are
all in normal position and
$\varphi_{0},\dots,\varphi_{n}=\Phi_{0},\dots,\Phi_{n}\bmod p$ are the
normalized Chow forms of these equidimensional components. Since
$\delta$ does not vanish modulo $p$,
$V(g_1,\dots,g_s)=V(f_1,\dots,f_s)^\ma$, with $f_i =F_i \bmod p$ for
all $i$ and $\ma = \mA \bmod p$.

Let $\widetilde{\mA}$ be the $(n+1)\times(n+1)$ matrix given by
$\widetilde{\mA}=[1] \oplus \mA^T$, and let $\widetilde{\ma}$ be its
reduction modulo $p$. We can then define the following polynomials:
\begin{itemize}
\item For $i=0,\dots,n$, set $C_i = \Phi_i^{\widetilde{\mA}^{\oplus
    {i+1}}}$, where $\widetilde{\mA}^{\oplus {i+1}}=
  \widetilde{\mA}\oplus \cdots \oplus \widetilde{\mA}$ ($i+1$ times)
  acts on the $i+1$ blocs of $n+1$ variables in $\Phi_i$. It follows
  that $C_0,\dots,C_n$ are Chow forms of the equidimensional
  components $V_0,\dots,V_n$ of $V(F_1,\dots,F_s)$, and still have
  coefficients in $\Z_{(p)}$; besides, for all $i$ we can write $C_i =
  C^{\rm primitive}_i/\alpha_i$, for some integer $\alpha_i$, where
  $C^{\rm primitive}_i$ is a primitive Chow form of $V_i$ (recall that
  these are defined up to sign).
\item Similarly, for $i=0,\dots,n$, we set $c_i = \varphi_i^{\widetilde{\ma}^{\oplus
    {i+1}}}$, with as above $\widetilde{\ma}^{\oplus {i+1}}=
  \widetilde{\ma}\oplus \cdots \oplus \widetilde{\ma}$. Since $\varphi_i$ is
  a Chow form of the $i$th equidimensional component of $V(f_1,\dots,f_s)^\ma$,
  it follows as above that $c_i$ is
  a Chow form of the $i$th equidimensional component $v_i$ of $V(f_1,\dots,f_s)$.
\end{itemize}
Since $c_i = C_i \bmod p$, the previous discussion gives $C^{\rm
  primitive}_i \bmod p = \alpha_i c_i$. On the other hand, all
$\alpha_i$ are nonzero modulo $p$, so that $\alpha_i c_i$ is a Chow
form of $v_i$ for all $i$. This means that as soon as $p$ does not
divide $\Delta'$, the conclusion of Theorem~\ref{theo:main} are
satisfied.


\section{Experiments}\label{sec:experiments}

For readability, the main result of Theorem \ref{theo:main}, and the
intermediate steps toward it, give upper bounds for height in
asymptotic notation. It is possible to derive inequalities for all
these steps, but writing and manually checking these inequalities
quickly becomes complex. Instead, we wrote a Maple script that
computes actual upper bounds, following the overall structure of the
proof. Here we give numerical values in a number of examples.

As mentioned in the introduction, D'Andrea et al. in
\cite{d2019reductions} give a bound similar to the one in our main
theorem, which is valid for systems which are known to have finitely
many solutions in $\C^n$ (in particular, if the number of given
equations $s$ is less than $n$, these are systems with no
solutions). We denote this bound by $A$, and compare with our bound
derived from $\Delta'$ as defined in the previous section, which we
denote as $B$. For each choice of parameters $(n,s,d,h)$ we compute
the numerical values of $A$ and $B$ and also report the ratio $B/A$
for comparison. All experiments were carried out in Maple (see \href{https://github.com/Jesse-Allister-Kasien-Elliott/Primes-of-bad-reduction-for-systems-of-polynomial-equations}{repository}).

We perform a sweep through each parameter, varying the parameter while
holding the others fixed.  In Table \ref{tab:sweep-n}, we sweep
through $n$ while fixing $s=5, d=2,$ and $h=10$. As $n$ increases, the
bounds grow more quickly than with the other parameters. In Table
\ref{tab:sweep-s} we sweep through $s$ while fixing $n=5, d=2,$ and
$h=10$. In Table \ref{tab:sweep-d} we sweep through $d$ while fixing
$n=3, s=5,$ and $h=10$. Lastly, in Table \ref{tab:sweep-h}, we sweep
through $h$ while fixing $n=3, s=5,$ and $d=3$.

\begin{table}[h!]
\centering
\begin{tabular}{c|c|c|c}
$n$ & $A$ & $B$ & $B/A$ \\
\hline
1  & $1.9921{\times}10^{4}$  & $3.4824{\times}10^{6}$  & $1.7481{\times}10^{2}$ \\
2  & $2.3695{\times}10^{5}$  & $4.3930{\times}10^{8}$  & $1.8540{\times}10^{3}$ \\
5  & $2.5031{\times}10^{8}$  & $3.3880{\times}10^{13}$ & $1.3535{\times}10^{5}$ \\
10 & $1.5907{\times}10^{13}$ & $1.2561{\times}10^{20}$ & $7.8968{\times}10^{6}$ \\
20 & $3.5113{\times}10^{22}$ & $2.2805{\times}10^{31}$ & $6.4947{\times}10^{8}$ \\
50 & $1.1769{\times}10^{50}$ & $3.4204{\times}10^{61}$ & $2.9063{\times}10^{11}$
\end{tabular}
\caption{Sweep in the number of variables $n$. The remaining parameters are fixed with $s=5$, $d=2$, $h=10$.}
\label{tab:sweep-n} 
\end{table}

\begin{table}[h!]
\centering
\begin{tabular}{c|c|c|c}
$s$ & $A$ & $B$ & $B/A$ \\
\hline
2   & $2.0540{\times}10^{8}$ & $1.3552{\times}10^{13}$ & $6.5979{\times}10^{4}$ \\
5   & $2.5031{\times}10^{8}$ & $3.3880{\times}10^{13}$ & $1.3535{\times}10^{5}$ \\
10  & $2.8429{\times}10^{8}$ & $6.7760{\times}10^{13}$ & $2.3834{\times}10^{5}$ \\
20  & $3.1827{\times}10^{8}$ & $1.3552{\times}10^{14}$ & $4.2580{\times}10^{5}$ \\
50  & $3.6319{\times}10^{8}$ & $3.3880{\times}10^{14}$ & $9.3284{\times}10^{5}$ \\
100 & $3.9717{\times}10^{8}$ & $6.7760{\times}10^{14}$ & $1.7061{\times}10^{6}$
\end{tabular}
\caption{Sweep in the number of equations $s$. The remaining parameters are fixed with $n=5$, $d=2$, $h=10$.}
\label{tab:sweep-s} 
\end{table}

\begin{table}[h!]
\centering
\begin{tabular}{c|c|c|c}
$d$ & $A$ & $B$ & $B/A$ \\
\hline
2  & $2.5455{\times}10^{6}$  & $2.6014{\times}10^{10}$ & $1.0219{\times}10^{4}$ \\
3  & $2.0926{\times}10^{8}$  & $2.8950{\times}10^{12}$ & $1.3835{\times}10^{4}$ \\
5  & $5.5270{\times}10^{10}$ & $1.3592{\times}10^{15}$ & $2.4593{\times}10^{4}$ \\
10 & $1.0949{\times}10^{14}$ & $6.0890{\times}10^{18}$ & $5.5611{\times}10^{4}$ \\
20 & $2.2045{\times}10^{17}$ & $2.7274{\times}10^{22}$ & $1.2372{\times}10^{5}$
\end{tabular}
\caption{Sweep in the degree $d$. The remaining parameters are fixed with $n=3$, $s=5$, $h=10$.}
\label{tab:sweep-d} 
\end{table}

\begin{table}[h!]
\centering
\begin{tabular}{c|c|c|c}
$h$ & $A$ & $B$ & $B/A$ \\
\hline
1   & $1.8959{\times}10^{8}$ & $1.6600{\times}10^{12}$ & $8.7556{\times}10^{3}$ \\
2   & $1.9178{\times}10^{8}$ & $1.7972{\times}10^{12}$ & $9.3713{\times}10^{3}$ \\
5   & $1.9833{\times}10^{8}$ & $2.2089{\times}10^{12}$ & $1.1137{\times}10^{4}$ \\
10  & $2.0926{\times}10^{8}$ & $2.8950{\times}10^{12}$ & $1.3835{\times}10^{4}$ \\
100 & $4.0589{\times}10^{8}$ & $1.5245{\times}10^{13}$ & $3.7559{\times}10^{4}$
\end{tabular}
\caption{Sweep in the coefficient height $h$. The remaining parameters are fixed with $n=3$, $s=5$, $d=3$.}
\label{tab:sweep-h} 
\end{table}

The bounds in both papers involve a dominant term $d^{3n}$, multiplied
by a quantity polynomial in $nds$ and linear in $h$, so as expected,
the dependence on the number of variables $n$ is the most pronounced.
In the examples tested, $B$ ranges from about $10^{6}$ (for small~$n$)
up to about $10^{61}$ (for the largest~$n$), while increases in $s$,
$d$, and $h$ produce more moderate changes.  The ratio $B/A$ varies
between approximately $10^{2}$ and $10^{11}$ across all experiments.

\paragraph*{Acknowledgements.} \'Eric Schost acknowledges
support from the Natural Sciences and Engineering Research Council of Canada
(NSERC) through a Discovery Grant.  Jesse Elliott acknowledges support by Inria through the OURAGAN project-team.

\bibliographystyle{plain} \bibliography{primes}

\end{document}